\documentclass[11pt, a4paper]{article}

\usepackage{amsmath, amsthm, amsfonts}

\usepackage{fullpage}

\theoremstyle{plain}
\newtheorem{theorem}{Theorem}
\newtheorem{corollary}[theorem]{Corollary}

\theoremstyle{definition}
\newtheorem*{remark}{Remark}

\usepackage{graphicx}

\usepackage{hyperref}

\newcommand{\N}{\mathbb{N}}
\newcommand{\Z}{\mathbb{Z}}
\newcommand{\Q}{\mathbb{Q}}
\newcommand{\R}{\mathbb{R}}
\newcommand{\C}{\mathbb{C}}
\newcommand{\K}{\mathbb{K}}

\usepackage[np]{numprint}
\npthousandsep{\,}
\npdecimalsign{.}
\newcommand{\ff}[2]{f^{\np{#1}}_{\np{#2}}}

\newcommand{\numer}{\operatorname{numer}}
\newcommand{\denom}{\operatorname{denom}}
\newcommand{\sgn}{\operatorname{sgn}}
\renewcommand{\Re}{\operatorname{Re}}
\renewcommand{\Im}{\operatorname{Im}}

\title{Three essays on Machin's type formulas%
\thanks{\textit{Indagationes Mathematicae}, to appear.}}

\author{Armengol Gasull\textsuperscript{1}, 
   Florian Luca\textsuperscript{2} 
   and Juan L. Varona\textsuperscript{3}
   \\[6pt]
   \small\textsuperscript{1}Dep. de Matem\`{a}tiques,
   Universitat Aut\`{o}noma de Barcelona and Centre de Recerca Matem\`{a}tica,\\[-2pt]
   \small Cerdanyola del Vall\`{e}s (Barcelona), Spain. 
   Email: {Armengol.Gasull@uab.cat}\\[6pt]
   \small\textsuperscript{2}School of Maths, Wits University, Johannesburg, South Africa,\\[-2pt]
   \small Centro de Ciencias Matem\'aticas, UNAM, Morelia, Mexico. 
   Email: {florian.luca@wits.ac.za}\\[6pt]
   \small\textsuperscript{3}Departamento de Matem\'aticas y Computaci\'on, 
   Universidad de La Rioja,\\[-2pt]
   \small Logro\~no, Spain. 
   Email: {jvarona@unirioja.es}}

\date{}

\begin{document}

\maketitle

\begin{abstract}
We study three questions related to Machin's type formulas. The first one gives all two terms Machin formulas where both arctangent functions are evaluated $2$-integers, that is values of the form $b/2^a$ for some integers $a$ and~$b$. These formulas are computationally useful because multiplication or division by a power of two is a very fast operation for most computers. The second one presents a method for finding infinitely many formulas with $N$ terms. In the particular case $N=2$ the method is quite useful. It recovers most known formulas, gives some new ones, and allows to prove, in an easy way, that there are two terms Machin formulas with Lehmer measure as small as desired. Finally, we correct an oversight from previous result and give all Machin's type formulas with two terms involving arctangents of powers of the golden section.
 
\medskip
 
\noindent
\textbf{Keywords:} 
Machin's type formulas, Lehmer measure, computation of $\pi$, golden section.
 
\medskip
 
\noindent
\textbf{Mathematics Subject Classification:} 
Primary 11Y60; Secondary 11D45, 33B10.
\end{abstract}

\section{Introduction}

In 1706, John Machin found the identity
\begin{equation}
\label{eq:Machin-intro}
  4 \arctan\frac{1}{5} - \arctan\frac{1}{239} = \frac{\pi}{4}.
\end{equation}
In conjunction with the $\arctan$ expansion
\begin{equation}
\label{eq:Gregory}
  \arctan x = \sum_{m=0}^{\infty} \frac{(-1)^m}{2m+1} x^{2m+1},
  \qquad |x| < 1,
\end{equation}
discovered by Gregory in 1671, Machin used \eqref{eq:Machin-intro} to compute 100 digits of~$\pi$.

In the mathematical literature there are many formulas similar to~\eqref{eq:Machin-intro}, that is, combinations of $\arctan$ functions that, in some way, generate~$\pi$. 
Besides \eqref{eq:Machin-intro}, the following are the most classical formulas 
\begin{align}
\label{eq:Euler}
& \arctan(1/2) + \arctan(1/3) = \pi/4,
\\
\label{eq:Hermann}
& 2\arctan(1/2) - \arctan(1/7) = \pi/4,
\\
\label{eq:Hutton}
& 2\arctan(1/3) + \arctan(1/7) = \pi/4,
\end{align}
that are usually known as Euler's, Hermann's and Hutton's formulas, respectively
(actually, their attribution to these authors is not clear; for instance, \cite{Cal2} also attributes all of them to Machin, see \cite{Tw} for more historical details).

While many of these formulas have been used to effectively compute many digits of $\pi$, other formulas do not have such practical interest, but they are interesting by themselves. For instance, this is the case with the relation
\begin{equation}
\label{eq:Fib}
\arctan \frac{F_{n}}{F_{n+1}} + \arctan \frac{F_{n-1}}{F_{n+2}} = \frac{\pi}{4},
\end{equation}
where $(F_n)_n$ are the Fibonacci numbers (a simple geometric proof of this formula can be found in~\cite{Plaza}). 
Moreover, taking into account that when $n \to \infty$, $F_n/F_{n-1} \to \phi := (1+\sqrt{5})/2$, the golden section, taking limits in \eqref{eq:Fib} gives the identity
\begin{equation}
\label{eq:phi}
\arctan \phi^{-1} + \arctan \phi^{-3} = \frac{\pi}{4}.
\end{equation}

Many questions can be posed around this subject. A first natural one was: How many formulas of the type
\begin{equation}
\label{eq:1/mk}
  x_1 \arctan\frac{1}{m_1} + x_2 \arctan\frac{1}{m_2} = \frac{\pi}{4},
\end{equation}
with rationals $x_k$ and integers $m_k \ge 2$ there exist? 
Nowadays, after 1895 St\"ormer's paper \cite{StormerOld} (see also \cite{Stormer}) it is known that 
only the four above identities \eqref{eq:Machin-intro}, \eqref{eq:Euler}, \eqref{eq:Hermann} and \eqref{eq:Hutton} do exist. 

How about if we allow identities of the type
\begin{equation}
\label{eq:Machin-like}
  x_1 \arctan\frac{a_1}{b_1} + \dots + x_N \arctan\frac{a_N}{b_N} = \frac{\pi}{4},
\end{equation}
with $x_k \in \Q$, $a_k \in \Z$, $b_k \in \N^*$ (and $|a_k/b_k| < 1$ to guarantee the convergence of \eqref{eq:Gregory} with $x=a_k/b_k$), are there many other such formulas?
Which of them gives a faster algorithm to compute digits of~$\pi$? 
In 1938, D. H. Lehmer \cite{Le} gave the now so-called Lehmer measure 
\begin{equation}
\label{eq:Lehmer-orig}
  \sum_{k=1}^N \frac{1}{\log_{10}(|b_k/a_k|)},
\end{equation}
that can be used as a hint of the computational efficiency of~\eqref{eq:Machin-like};
without explaining the details that motivate the definition, note that, if $|a_k/b_k|$ is small, the series \eqref{eq:Gregory} for $\arctan(a_k/b_k)$ converges quickly, and less summands are necessary to compute it with a prescribed precision. Thus, the smaller is the Lehmer measure, the faster is the corresponding algorithm to compute digits of~$\pi$. 
Many formulas of type \eqref{eq:Machin-like} with their corresponding Lehmer measures can be found in \cite{ChLi, Le, We}. For instance, the Lehmer measure of~\eqref{eq:Machin-intro} is $1.85113$ and thus it is faster than~\eqref{eq:Euler}, \eqref{eq:Hermann} and~\eqref{eq:Hutton}, whose Lehmer measures are, respectively, $5.41783$, $4.50522$ and $3.2792$; 
moreover, both \cite{ChLi} and \cite{We} give the same identity of type \eqref{eq:Machin-like}, with $N=6$, and whose Lehmer measure is $1.51244$, the lowest at that time. Are there Machin-like formulas with Lehmer measure as small as we want?

Nowadays, the use of this type of formulas to compute many digits of $\pi$ is not so useful, because faster types of algorithms are available (for instance, Chudnovsky algorithm \cite{Chu}, which is based on Ramanujan's $\pi$ formulas; for more details on these types of algorithms see~\cite{Gu}). Actually, more than $10^{15}$ decimal digits of $\pi$ are already known. Moreover, to compute more digits of $\pi$ does not have any practical interest, but the one of beating records.

In relation to \eqref{eq:phi}, a different question can be asked: are there similar formulas with other powers of~$\phi$?

The aim of this paper is to answer some of the above questions. In Section~\ref{sec:powers2}, we analyze the solutions of an equation similar to \eqref{eq:1/mk} but allowing $\arctan(2^{a_k}/m_k)$ or $\arctan(m_k/2^{a_k})$ in the place of $\arctan(1/m_k)$. We prove that there are ten sporadic Machin-type formulas of this type, together with two parametric families, see Theorem~\ref{theo:power2}.

Let us comment why the interest of having $2^{a_k}/m_k$ or $m_k/2^{a_k}$ instead of $a_k/b_k$ in general (we assume here that $a_k,b_k,m_k$ are positive integers). Let us assume that we want to compute many summands in \eqref{eq:Gregory}, to get many digits of~$\pi$. If we have $x=1/m_k$, every summand requires to divide by $2m+1$ and by $m_k^2$ (two operations); if we have $x=a_k/b_k$, a division by $2m+1$ and by $b_k^2$ and a multiplication by $a_k^2$ (three operations). 
Due to this reason, most of the Machin-like formulas to compute $\pi$ that have been used in the practice (or whose Lehmer measure have been analyzed in the above mentioned papers~\cite{ChLi,Le, We}) are of the form $1/m_k$.
But, if we have $2^{a_k}/m_k$ or $m_k/2^{a_k}$, to multiply or to divide by $2^{a_k}$ can be done with a \textit{shift} in the binary representation of the number, whose computational time is negligible compared with a multiplication or a division, so this case can be considered as fast as the case with $1/m_k$.
Thus, perhaps a better way to estimate the computational efficiency of a formula like~\eqref{eq:Machin-like} would be to take
\begin{equation*}
\sum_{k=1}^N \frac{w_k}{\log_{10}(|b_k/a_k|)}
\end{equation*}
with some ``weights'' $w_k \ge 1$ that may depend on $a_k$, and~$b_k$ (as well as on the hardware and the software), so it not totally clear how to compare such formulas.

In Section~\ref{sec:R_j} we define some rational functions $R_j(n,x)$ (both the numerator and the denominator being polynomials in the variable $x$ depending on $n$ and with integer coefficients), $j=0,1,2,3$ and $n \in \N$, in such a way that, for any $x \in \Q$, the combinations
\[
  x_1 \arctan(R_{j_1}(n_1,x)) + \cdots + x_N \arctan(R_{j_N}(n_N,x)),
\]
with $x_k=r_k/n_k$ and $r_1 + \cdots + r_N = 0$, always give a rational multiple of $\pi$ (we ignore the poles, namely the values of $x$ that are roots of any denominator). We have used the name ``Machin's formulas machine'' to denominate this method, because it allows finding Machin's type formulas without any difficulty. In particular, taking $N=2$, it allows us to find Machin's type formula with Lehmer measure as small as we want, see Theorem~\ref{theo:mLeps}. As we will comment at the end of Subsection~\ref{subsec:ss} our Machin's type formulas when $N=2$ extend some of the results of~\cite{ABCM}.

Finally, in Section~\ref{sec:phi}, we classify the formulas of the type
\[
  x_1 \arctan(\phi^{a_1}) + x_2 \arctan(\phi^{a_2}) = \frac{\pi}{4},
\]
with $a_k \in \Z \setminus \{0\}$ and $x_k \in \Q \setminus \{0\}$, showing that there are, essentially, sixteen of these identities. In fact, this part is a correction of the previous paper~\cite{LuSt} where some of these formulas were missed due to an oversight in the proof.

\section{Machin's formulas with powers of two}
\label{sec:powers2}

The purpose of this section is to solve
\begin{equation}
\label{eq:0}
x_1 \arctan(z_1) + x_2 \arctan(z_2) = \frac{\pi}{4}
\end{equation}
in rational numbers $x_1,x_2,z_1,z_2$, where $z_k\in (0,1)$ for $k=1,2$ and $z_k = 2^{a_k}/b_k$ or $b_k/2^{a_k}$ for some integers $a_k, b_k\ge 1$. The case where $a_k\le 0$ for both $k=1,2$ leads to $z_k=1/m_k$ 
for $k=1,2$, and this has been treated \cite{Stormer}. We do not treat the case when $z_1=z_2=z$, since that leads to $\arctan(z)/\pi\in \Q \setminus\{0\}$, and the only corresponding value of $z$ is~$1$. So, we assume that $z_1<z_2$. 
In case $a_k\le 0$, we incorporate $2^{-a_k}$ into $b_k$. Hence, we assume that $a_k\ge 0$ and $b_k$ is odd unless $a_k=0$ in which case $b_k$ can be even. 

As we will see, a main tool in the proof of next theorem will be that all positive integer solutions $(x,y,a,n)$, $n\ge3$, of the diophantine equations
$x^2+1 = 2y^n$ and $x^2+2^a = y^n$, are known, see \cite{Cohn, Lu}.

\begin{theorem}
\label{theo:power2}
All solutions $(x_1,z_1,x_2,z_2)$ of equation \eqref{eq:0} in non-zero rational numbers $x_1,x_2$ and rational numbers $z_1<z_2$ in $(0,1)$ of the form $2^{a_k}/b_k$ or $b_k/2^{a_k}$ for $k=1,2$ are the following ten sporadic ones
\begin{align*}
& \left(-1, \frac{1}{239}, 4, \frac{1}{5}\right), \ 
\left(-1, \frac{1}{7}, 2, \frac{1}{2}\right), \ 
\left(-1, \frac{2}{11}, \frac{3}{2}, \frac{3}{4}\right), \\
& \left(-1, \frac{2}{11}, 3, \frac{1}{3}\right), \ 
\left(\frac{1}{3}, \frac{1}{239}, \frac{4}{3}, \frac{2}{3}\right), \ 
\left(\frac{1}{2}, \frac{2}{11}, \frac{3}{2}, \frac{1}{2}\right), \\ 
& \left(1, \frac{1}{41}, 2, \frac{2}{5}\right), \ 
\left(1, \frac{1}{7}, 2, \frac{1}{3}\right), \ 
\left(1, \frac{1}{2}, \frac{1}{2}, \frac{3}{4}\right), \ 
\left(3, \frac{1}{7}, 2, \frac{2}{11}\right),
\end{align*}
together with the two parametric families
\[
\left(1,\frac{1}{2^{a_2+1}+1},1,\frac{2^{a_2}}{2^{a_2}+1}\right), \quad 
\left(1,\frac{1}{2^{a_2+1}-1},1,\frac{2^{a_2}-1}{2^{a_2}}\right),
\qquad a_2\in \N^*.
\]
\end{theorem}

\begin{remark}
Allowing $a_2=0$ in the first parametric family we get the solution $\left(1,\frac{1}{3},1,\frac{1}{2}\right)$ which also belongs to the second parametric family 
(for $a_2=1$). We cannot allow $a_2=0$ in the second family because $z_2$ vanishes in this case. 
\end{remark}

\subsection{A reformulation}
 
We write $x_k=u_k/(w/d_0)$, where $u_1,u_2,d_0\ge 1$ are integers with $|u_1|,|u_2|,d_0,w\ge 1$ and $\gcd(u_1,u_2)=1$ and so
\begin{equation}
\label{eq:u1u2}
u_1\arctan(z_1)+u_2\arctan(z_2) = \frac{w\pi}{4d_0} = \frac{c\pi}{d}.
\end{equation}
Formally, we write first $x_1=U_1/w$, $x_2=U_2/w$, with a common denominator $w$, then let $d_0 = \gcd(U_1,U_2)$, so $u_1 = U_1/d_0$, $u_2=U_2/d_0$. 
We write $w/(4d_0)=c/d\ne 0$ in reduced terms. Applying $\tan$, we get that $\tan(c\pi/d)\in \Q$. In particular, this implies that $e^{2c i \pi/d}\in \Q[i]$, so $e^{2ci \pi/d}$ is a root of unity of order at most~$2$. 
This implies that $\varphi(d)\le 2$, so $d\in \{1,2,3,4,6\}$, where $\varphi$ is the Euler's totient function. In fact, by using \cite[Cor.~3]{Cal2}, it can be seen that $d\in \{1,2,4\}$, but we will not use this fact because this stronger restriction does not imply substantial changes in our proof and in this way our argument is more self-contained.
To fix notations, we assume that $a_1\le a_2$ and we do not consider the case $a_1=a_2=0$, since those solutions have already been found in~\cite{Stormer}. Thus, $a_2\ge 1$. 

\subsection{Proof of Theorem~\ref{theo:power2}}

Assume for the sake of the argument that $z_1=2^{a_1}/b_1$, $z_2=2^{a_2}/b_2$. The cases where $z_k=b_k/2^{a_k}$ for one or both of $k=1,2$, can be reduced to the present one via the formula 
\[
\arctan\left(\frac{1}{x}\right) = \frac{\pi}{2}-\arctan(x),
\]
arriving to an equation similar to \eqref{eq:u1u2} with a different value of $c/d$ in the right-hand side. Up to replacing $(u_1,u_2)$ by 
$(-u_1,-u_2)$ if needed, we assume that $u_1\ge 1$. Noting that $d\mid 12$, it follows that $12/d\in \N$. Next, we get
\[
(1+i2^{a_1}/b_1)^{12u_1} (1+i 2^{a_2}/b_2)^{12u_2}
= (1-i2^{a_1}/b_1)^{12u_1} (1-i 2^{a_2}/b_2)^{12u_2}.
\]
Thus, 
\[
(b_1+i2^{a_1})^{12u_1} (b_2\pm i 2^{a_2})^{12|u_2|}
= (b_1-i 2^{a_1})^{12u_1} (b_2\mp i2^{a_2})^{12|u_2|},
\]
where the sign in $\pm $ on the left is $\sgn(u_2)$ (and the sign in $\mp$ on the right is $-\sgn(u_2)$). Extracting $12$th roots we get
\[
(b_1+i2^{a_1})^{u_1} (b_2\pm i 2^{a_2})^{|u_2|} 
= \zeta(b_1-i 2^{a_1})^{u_1} (b_2\mp i2^{a_2})^{|u_2|},
\]
where $\zeta$ is a root of unity in $\Q[i]$. Hence, $\zeta\in \{\pm 1,\pm i\}$. 

\subsubsection{The case $a_1\ge 1$}

Assume first that $a_1\ge 1$. Then $b_1+2^{a_1} i$ and $b_1-2^{a_1} i$ are coprime in $\Z[i]$ since their norms are $b_1^2+2^{2a_1}$ (odd) but the norm of their difference $2^{a_1+1} i$ is a power of $2$ and the same is true about 
$b_2+2^{a_2} i$ and $b_2-2^{a_2} i$. It follows up to relabelling $\zeta$ that 
\[
(b_1 + 2^{a_1} i)^{u_1} = \zeta (b_2 \mp 2^{a_2} i)^{|u_2|},
\]
where $\zeta$ is unit in $\Z[i]$. Thus, $\zeta\in \{\pm 1,\pm i\}$. If $u_1=|u_2|$, then $1=u_1=|u_2|$ (since they are coprime) so $b_1+2^{a_1} i= \zeta(b_2\mp 2^{a_2} i)$, so we get $b_1=b_2, a_1=a_2$, so $z_1=z_2$, a case that we do not consider. 
So, we assume that $u_1\ne |u_2|$. Then there exists $\gamma\in \Z[i]$ such that 
\[
b_1+2^{a_1} i = \zeta_1 \gamma^{|u_2|}
\quad \text{and}\quad 
b_2\mp 2^{a_2} i = \zeta_2 \gamma^{u_1},
\]
where again $\zeta_1,\zeta_2$ are in $\{\pm 1,\pm i\}$. Assume next that $\{u_1,|u_2|\}=\{1,2\}$. Swapping $u_1$ and $|u_2|$ if needed and incorporating $\zeta_1$ into $\gamma$ we get
\[
b_1+2^{a_1}i = \gamma
\quad \text{and}\quad 
b_2 \mp 2^{a_2} i = \pm \zeta_2 \gamma^2.
\]
Thus, 
\[
b_2\mp 2^{a_2} i = \zeta_2 (b_1+2^{a_1} i)^2 = \zeta_2(b_1^2-2^{2a_1}+2^{a_1+1}b_1 i),
\quad \zeta_2\in \{\pm 1,\pm i\}.
\]
Since $b_2$ and $b_1^2-2^{2a_1}$ are both odd, we get that $\zeta_2\in \{\pm 1\}$, $2^{a_2}=2^{a_1+1} b_1$ and $b_2=\pm (b_1^2-2^{2a_1})$. The first equation leads to $a_2=a_1+1$, $b_1=1$, and now the second leads to 
$b_2=\pm (1^2-2^{2a_1})$, so $b_2=2^{2a_1}-1$. This leads to
\[
2\arctan\left(\frac{1}{2^{a_1}}\right) - \arctan\left(\frac{2^{a_1+1}}{2^{2a_1}-1}\right) = 0,
\]
which follows from the well-known formula
\[
2\arctan(x) = \arctan\left(\frac{2x}{1-x^2}\right),
\]
for $x\in (-1,1)$, with $x=1/2^{a_1}$. 
When $a_1=1$ the above formula gives rise to the solution $(1,\frac{1}{2},\frac{1}{2},\frac{3}{4})$. 
For $a_1>1$ this looks like \eqref{eq:u1u2} except that it has $c/d=0$, which is not convenient for~us. 

This was for $a_1\ge 1$ and $\{u_1,|u_2|\} = \{1,2\}$. Up to swapping $u_1,u_2$, we next assume that $|u_2|\ge 3$. Then
\[
b_1+2^{a_1} i = \zeta_1 \gamma^{|u_2|}.
\]
Taking norms we get
\[
b_1^2+2^{2a_1}=y^{|u_2|}.
\]
The solutions of the equation
\[
x^2+2^a=y^n,
\]
with $x$ odd and $n\ge 3$, have been found in~\cite{Lu}. They are 
\[
5^2+2=3^3,\quad 11^2+2^2=5^3,\quad 7^2+2^5=3^4.
\]
Only the second one is convenient for us (the exponent of $2$ must be even) giving $b_1=11$, $a_1=1$, $u_2=\pm 3$. Thus, $\gamma=1+2i$ and $u_1\in \{1,2\}$. Hence, we must also have 
\[
b_2\mp 2^{a_2} i=\zeta_2 (1\pm 2i)^{1,2} \in \{\zeta_2(1\pm 2i), \zeta_2(-3\pm 4i)\}.
\]
Thus, we get $\zeta_2\in \{\pm 1\}$, $(b_2,a_2)\in \{(1,1),(3,2)\}$.

\subsubsection{The case $a_1=0$}

In case $b_1$ is even, the same arguments apply because $b_1+i$ and $b_1-i$ are coprime since their norm is $b_1^2+1$ (odd) and the norm of their difference $2i$ is $4$ which is a power of~$2$. The previous arguments apply. We get 
$(u_1,u_2)=(1,\pm 1)$ and $(b_1,a_1)=(b_2,a_2)$ which leads $z_1=z_2$ which is not convenient. The case $(u_1,|u_2|)=(1,2)$ does not lead to convenient solutions since $b_2=2^{2a_1}-1=0$, which is not possible. 
The case $\max\{u_1,|u_2|\}\ge 3$, leads again to $x^2+2^{2a}=y^n$, where $(x,a)=(b_k,a_k)$ for some $k \in \{1,2\}$. This equation has no solution with $a=0$, so we get $(b_2,a_2,u_1)=(11,1,3)$. Hence, 
$|u_2|\in \{1,2\}$, $a_1=0$, $\gamma=1+2i$, so 
$b_1+i = \zeta_1\gamma^{1,2}\in \{\zeta_1(1+2i),\zeta_1(-3+4i)\}$, 
so the only possibility is $b_1=2$, $u_2=\pm 1$.

Finally suppose that $a_1=0$, $b_1$ is odd. In this case in 
\[
(b_1+i)^{4u_1} (b_2\pm 2^{a_2} i)^{4|u_2|} = (b_1-i)^{4u_1} (b_2\mp 2^{a_2} i)^{4|u_2|},
\]
we have that $b_1+i$ has norm $b_1^2+1\equiv 2\pmod 8$. Thus, $1+i\mid b_1+i$ and $(b_1+i)/(1+i)$ is an integer in $\Z[i]$ of odd norm. Thus,
\[
\left(\frac{b_1+i}{1+i}\right)^{4u_1} (b_2\pm 2^{a_2} i)^{4|u_2|}
= \left(\frac{b_1-i}{1-i}\right)^{4u_1} (b_2\mp 2^{a_2} i)^{4|u_2|}.
\]
Now the integer $(b_1+i)/(1+i)$ is coprime to $(b_1-i)/(1-i)$ (since they have odd norms and $2$ is linear combination of the above two integers with coefficients in $\Z[i]$), so we get that 
\[
\left(\frac{b_1+i}{1+i}\right)^{4u_1} = \zeta (b_2\mp 2^{a_2} i)^{4|u_2|},
\]
for some unit $\zeta$ in $\Z[i]$. Thus, there is $\gamma\in \Z[i]$ and two units $\zeta_1,\zeta_2$ such that 
\begin{equation}
\label{eq:con}
\frac{b_1+i}{1+i}=\zeta_1 \gamma^{|u_2|}
\quad \text{and}\quad 
b_2\mp 2^{a_2} i=\zeta_2 \gamma^{u_1}.
\end{equation}
If $u_1=|u_2|$, then $u_1=|u_2|=1$. In this case we get
\[
b_1+i=\zeta(1+i)(b_2\mp 2^{a_2} i) = \zeta(b_2\pm 2^{a_2}+i(b_2\mp 2^{a_2})),
\quad 
\zeta\in \{\pm 1,\pm i\}.
\]
We study the four possibilities. If $\zeta=\pm 1$, we then get
\[
b_1+i = \pm (b_2\pm 2^{a_2}+i(b_2\mp 2^{a_2})).
\]
This gives 
\[
b_1 = \pm (b_2\pm 2^{a_2}),\quad 1=\pm (b_2\mp 2^{a_2}),
\]
which correspond to the systems
\[
\left\{\begin{aligned} b_1 & = b_2 + 2^{a_2}, \\ 
1 & = b_2 - 2^{a_2}, \end{aligned}\right. \quad 
\left\{\begin{aligned} b_1 & = b_2 - 2^{a_2}, \\ 
1 & = b_2 + 2^{a_2}, \end{aligned}\right. \quad 
\left\{\begin{aligned} b_1 & = -(b_2 + 2^{a_2}), \\ 
1 & = -(b_2 - 2^{a_2}), \end{aligned}\right. \quad 
\left\{\begin{aligned} b_1 & = -(b_2 - 2^{a_2}), \\ 
1 & = -(b_2 + 2^{a_2}). \end{aligned}\right.
\]
Only the first system gives the acceptable solution $b_2=2^{a_2}+1$, $b_1=b_2+2^{a_2}=2^{a_2+1}+1$ yielding the first parametric family together with the solution with $a_2=0$, which is $\left(1,\frac{1}{3},1,\frac{1}{2}\right)$ and which is also a member of the second parametric family. 
The other three systems do not give acceptable solutions since one (or both) of $b_1$,~$b_2$ are negative.  
Assume next that $\zeta=\pm i$. We obtain 
\[
b_1 = \mp (b_2\mp 2^{a_2}),\qquad 1 = \pm (b_2\pm 2^{a_2}),
\]
which correspond to the systems
\[
\left\{\begin{aligned} b_1 & = -b_2 + 2^{a_2}, \\ 
1 & = b_2 + 2^{a_2}, \end{aligned}\right. \quad
\left\{\begin{aligned} b_1 & = -b_2 - 2^{a_2}, \\ 
1 & = b_2 - 2^{a_2}, \end{aligned}\right. \quad 
\left\{\begin{aligned} b_1 & = b_2 - 2^{a_2}, \\ 
1 & = - b_2 - 2^{a_2}, \end{aligned}\right. \quad 
\left\{\begin{aligned} b_1 & = b_2 + 2^{a_2}, \\ 
1 & = -b_2 + 2^{a_2}, \end{aligned}\right.
\]
where only the last one gives the acceptable solution $b_2=2^{a_2}-1$, $b_1=b_2+2^{a_2}=2^{a_2+1}-1$. This yields the second parametric family, after using $\arctan(x)=\pi/2-\arctan(1/x)$ for $x>0$. 
Again the other three systems do not give convenient solutions since one or both of $b_1,b_2$ are negative. 

Assume next that $u_1\ne |u_2|$. If $(u_1,|u_2|) \in \{(2,1),(1,2)\}$, then we get equations 
\[
\frac{b_1+i}{1+i} = \gamma
\quad \text{and}\quad 
b_2\pm 2^{a_2} i = \zeta_2\gamma^2,
\]
or 
\[
b_2\pm 2^{a_2} i = \gamma 
\quad \text{and}\quad 
\frac{b_1+i}{1+i} = \zeta_1\gamma^2.
\]
In the first case, we get 
\[
b_2\pm 2^{a_2} i = \zeta\left(\frac{b_1+i}{1+i}\right)^2 = \frac{\zeta'}{2} (b_1^2-1+2b_1i)
\qquad (\text{with } \zeta':=-i\zeta),
\]
which gives $b_2=b_1$, $2^{a_2+1}=b_1^2-1$. The only solution of the last equation above is $a_2=2$, $b_1=b_2=3$. This leads to the useless formula
\[
2\arctan\left(\frac{1}{3}\right) - \arctan\left(\frac{3}{4}\right) = 0.
\]
In the second case, we get
\begin{align*}
b_1+i & = \zeta(1+i) (b_2\mp 2^{a_2} i)^2 = \zeta(1+i)(b_2^2-2^{2a_2}\mp 2^{a_2+1}b_2 i) \\
&= \zeta(b_2^2-2^{2a_2}\pm 2^{a_2+1}b_2+(b_2^2-2^{2a_2} \mp 2^{a_2+1}b_2)i).
\end{align*}
When $\zeta=\pm 1$, we find
\[
b_2^2-2^{2a_2}\pm 2^{a_2+1}b_2 = b_1, \qquad 
b_2^2-2^{2a_2}\mp 2^{a_2+1}b_2 = 1,
\]
or
\[
b_2^2-2^{2a_2}\pm 2^{a_2+1}b_2 = -b_1, \qquad 
b_2^2-2^{2a_2}\mp 2^{a_2+1}b_2 = -1.
\]
The first case gives rise to the system
\[
\left\{\begin{aligned} 
(b_2\pm 2^{a_2})^2-2^{2a_2+1} & = b_1,\\
(b_2\mp 2^{a_2})^2 & = 2^{2a_2+1}+1. 
\end{aligned}\right.
\]
This is solvable in integers only when $a_2=1$. In this case, we find
\[
\left\{\begin{aligned} 
(b_2+2)^2-8 & = b_1, \\
(b_2-2)^2 & = 9,
\end{aligned} \right. \qquad 
\left\{\begin{aligned} 
(b_2-2)^2-8 & = b_1, \\
(b_2+2)^2 & = 9, 
\end{aligned}\right.
\]
so from $(b_2-2)^2=9$, we have the only acceptable solution $b_2=5$, therefore $b_1=41$, while from $(b_2+2)^2=9$, we have the only acceptable solution $b_2=1$, but this leads to $b_1=-7$, which is not acceptable. On the other hand the second case corresponds to
\[
\left\{\begin{aligned} 
(b_2\pm 2^{a_2})^2-2^{2a_2+1} & = -b_1, \\
(b_2\mp 2^{a_2})^2 & = 2^{2a_2+1}-1, 
\end{aligned}\right.
\]
which is solvable in integers only when $a_2=0$. In this case we find
\[
\left\{\begin{aligned} (b_2+1)^2-2 & = -b_1, \\
(b_2-1)^2 & = 1,
\end{aligned} \right. \qquad 
\left\{\begin{aligned} (b_2-1)^2-2 & = - b_1, \\
(b_2+1)^2 & = 1,
\end{aligned}\right.
\]
so from $(b_2-1)^2=1$, we have the only acceptable solution $b_2=2$, so $b_1=-7$, which is not acceptable, while from $(b_2+1)^2=1$ we do not have acceptable solutions. 
Finally, when $\zeta=\pm i$, we find
\[
b_2^2-2^{2a_2}\mp 2^{a_2+1}b_2 = -b_1,\qquad 
b_2^2-2^{2a_2}\pm 2^{a_2+1}b_2 = 1,
\]
or
\[
b_2^2-2^{2a_2}\mp 2^{a_2+1}b_2 = b_1,\qquad 
b_2^2-2^{2a_2}\pm 2^{a_2+1}b_2 = -1.
\]
The first case gives rise to the system
\[
\left\{\begin{aligned} 
(b_2\mp 2^{a_2})^2-2^{2a_2+1} & = -b_1, \\
(b_2\pm 2^{a_2})^2 & = 2^{2a_2+1}+1, 
\end{aligned}\right.
\]
which is solvable in integers only when $a_2=1$. Accordingly, we find
\[
\left\{\begin{aligned} 
(b_2-2)^2-8 & = -b_1, \\
(b_2+2)^2 & = 9,
\end{aligned} \right. \qquad 
\left\{\begin{aligned} 
(b_2+2)^2-8 & = -b_1, \\
(b_2-2)^2 & = 9, 
\end{aligned}\right.
\]
so from $(b_2+2)^2=9$ we have the only acceptable solution $b_2=1$, therefore $b_1=7$, while from $(b_2-2)^2=9$ we have the only acceptable solution $b_2=5$ but this leads to $b_1=-41$, which is not acceptable. On the other hand the second case corresponds to 
\[
\left\{\begin{aligned} 
(b_2\pm 2^{a_2})^2-2^{2a_2+1} & = b_1, \\
(b_2\mp 2^{a_2})^2 & = 2^{2a_2+1}-1, 
\end{aligned}\right.
\]
which is solvable in integers only when $a_2=0$. In this case, we find
\[
\left\{\begin{aligned} 
(b_2+1)^2-2 & = b_1, \\
(b_2-1)^2 & = 1, 
\end{aligned} \right. \qquad 
\left\{\begin{aligned} 
(b_2-1)^2-2 & = b_1, \\
(b_2+1)^2 & = 1, 
\end{aligned}\right.
\]
so from $(b_2-1)^2=1$ we have the only acceptable solution $b_2=2$, therefore $b_1=7$, while from $(b_2+1)^2=1$, we do not have acceptable solutions. Resuming this discussion, we find 
\[
(a_2,b_1,b_2) \in \{(0,7,2),(1,7,1),(1,41,5)\}.
\]
The first two instances lead to the same sporadic solution $\left(-1,\frac{1}{7},2,\frac{1}{2}\right)$ as $2^{a_1}/b_1=1/7$ and $2^{a_2}/b_2\in \{1/2,2\}$, namely the second one in the list from the statement of the theorem, while the third instance leads to the seventh sporadic solution $\left(1,\frac{1}{41},2,\frac{2}{5}\right)$ from the statement of the theorem. 

Finally, assume that $\max\{u_1,|u_2|\}\ge 3$. In this case taking norms in \eqref{eq:con} we get 
\[
b_1^2+1 = 2y^{u_1}
\quad \text{and}\quad 
b_2^2+2^{2a_2} = y^{|u_2|}.
\]
If $|u_2|\ge 3$, then we saw before that $b_2=11$, $a_2=1$ are the only possibilities and then $y=5$. If this is so and $u_1\in \{1,2\}$, we get $b_1^2+1\in\{2\cdot 5,\,2\cdot 5^2\}$, so $b_1\in \{3,7\}$. 
Finally, if $u_1\ge 3$, then we get the equation 
\[
b_1^2+1 = 2y^n,
\]
for some $n\ge 3$, and the only solutions are 
$(b_1,y,n) \in \{ (1,1,n), (239,13,4)\}$ (see~\cite{Cohn}). The first one gives no solution for $b_2^2+2^{2a_2} = y^{|u_2|}=1.$ The second one gives $u_1=4$, $b_1=239$. If also $|u_2|\ge 3$, then $u_2=\pm 3$, $b_2=11$, $a_2=1$, otherwise $u_2\in \{\pm 1,\pm 2\},$ and 
\[
b_2\pm 2^{a_2} i = \zeta (3\pm 2i)^{1,2}\in \{\zeta(3\pm 2i), \zeta(5\pm 12i) \},
\]
and the only convenient one is $u_2\in \{\pm 1\}$, $b_2=3$, $a_2=1$. Collecting all the intermediary values, we get the theorem modulo checking for the solutions to \eqref{eq:u1u2} which come from values of the parameters $(u_1,u_2,a_1,b_1,a_2,b_2,c/d)$ in the ranges $|u_k|\le 4$, $a_k\in \{0,1,2\}$ for both $k=1,2$, $d\in \{1,2,3,4,6\}$, $0<|c|\le 24$ and $b_k\in \{1,2,3,5,7,11,41,239\}$ for $k=1,2$. Both Mathematica and Maple codes returned the ten listed sporadic solutions. 

\section{The Machin's formulas machine}
\label{sec:R_j}

An easy way to prove well-known formulas as $\arctan(x) + \arctan(1/x) = \sgn(x)\pi/2$ or 
\begin{equation*}
\arctan(x) - \frac{1}{2} \arctan\left(\frac{2x}{1-x^2}\right)
=
\begin{cases} 
\pi/2, & \text{if }x>1, \\
0, & \text{if }|x|<1, \\
-\pi/2, & \text{if }x<-1,
\end{cases}
\end{equation*}
or many others, is to use derivatives. For instance, we can check that the derivative of $\arctan\big(\frac{2x}{1-x^2}\big)$ coincides with $\frac{d}{dx}\arctan(x) = \frac{1}{1+x^2}$ except for a multiplicative constant, so a suitable linear combination of $\arctan(x)$ and $\arctan\big(\frac{2x}{1-x^2}\big)$ gives a function whose derivative is zero, therefore it is piecewise constant (namely it is constant except at the discontinuity points). 

More generally, we can find relations of the form
\[
  \arctan(x) + C \arctan(f(x)) = \text{constant}
\]
if we have functions $f(x)$ such that
\begin{equation}
\label{eq:atanf}
\frac{d}{dx}\arctan(f(x)) = \frac{r}{1+x^2},
\end{equation}
for some constant $r$. Furthermore, it is easy to check that 
\begin{equation}
\label{eq:compos}
\begin{aligned}
\frac{d}{dx}\arctan(f(x)) = \frac{r}{1+x^2},
\quad
&\frac{d}{dx}\arctan(g(x)) = \frac{s}{1+x^2}
\\
&\quad\Longrightarrow\quad
\frac{d}{dx}\arctan(g(f(x))) = \frac{rs}{1+x^2},
\end{aligned}
\end{equation}
so the composition of functions satisfying \eqref{eq:atanf} provides new examples.

For differentiable functions, \eqref{eq:atanf} is equivalent to solve the differential equation
\begin{equation}
\label{eq:ED}
\frac{f'(x)}{1+f(x)^2} = \frac{k}{1+x^2}.
\end{equation}
The solutions of this equation are
\begin{equation}
\label{eq:f(x)c}
  f(x) = \tan(k \arctan(x)+ c), 
\end{equation}
with $f(0) = \tan(c)$ and $c\in(-\pi/2,\pi/2)$. For our interest concerning Machin-like formulas, we want to have functions which are rational; that is, are ratios of polynomials with coefficients in~$\Z$.

Moreover, if we fix $c$ and denote the solution of \eqref{eq:ED} by $f_k$, the use of 
\[
\tan(a+b) = \frac{\tan a + \tan b}{1-(\tan a)(\tan b)}
\]
gives
\begin{equation}
\label{eq:fkrecu}
\begin{aligned}
f_{k+1}(x) &= \tan\big((k \arctan(x)+ c) + \arctan(x)) \\
&= \frac{\tan(k \arctan(x)+ c)+\tan(\arctan(x))}{1-\tan(k \arctan(x)+ c)\tan(\arctan(x))}
= \frac{f_{k}(x)+x}{1-xf_{k}(x)}.
\end{aligned}
\end{equation}
From this relation, if $f_1(x)$ is a rational function for a certain $c$, every function $f_k(x)$ will be rational. But $f_{1}(x) = (\tan(c)+x)/(1-x\tan(c))$, so we want that $\tan(c) \in \Q$ (or we start with $f_0(x)$ such that $f_0(x) = \tan c$, and we arrive at the same condition).

So, we take $c \in \pi\Q$. This is not compulsory, but it is suitable for our purposes. It is well known that $\tan(c) \in \Q$ if and only if $\tan(c) = 0$ or~$\pm 1$. 
Because $\tan$ is $(-\pi/2,\pi/2)$-periodic, we can restrict to one of the cases: $c=0$, $c=\pi/4$, $c=\pi/2$ and $c=-\pi/4$ (or $c=3\pi/4$, that will be more convenient notationwise).

The recurrence relation \eqref{eq:fkrecu} is nice and could be more widely studied, but here we are only interested in more explicit formulas.

\subsection{The functions $R_j(n,x)$}
\label{subsec:Rj}

Let us recall De Moivre's formula
\[
\cos (n \theta)+i \sin(n \theta) = (\cos (\theta)+i \sin(\theta))^{n}.
\]
Using the binomial expansion and equaling imaginary and real parts we get
\begin{align*}
  \sin(n \theta) &= \sum_{r=0}^{\lfloor(n-1)/2\rfloor} (-1)^{r} 
  \binom{n}{2r+1} \cos^{n-2 r-1}(\theta) \sin^{2r+1}(\theta)
  \\&= \cos^{n}(\theta) \sum_{r=0}^{\lfloor(n-1)/2\rfloor} (-1)^{r} 
  \binom{n}{2r+1} \tan^{2r+1}(\theta),
\end{align*}
\[
  \cos (n \theta) = \sum_{r=0}^{\lfloor n/2\rfloor} (-1)^{r} 
  \binom{n}{2r} \cos^{n-2 r}(\theta) \sin^{2 r}(\theta)
  = \cos^{n}(\theta) \sum_{r=0}^{\lfloor n/2\rfloor} (-1)^{r} 
  \binom{n}{2r} \tan^{2r}(\theta).
\]
And, by dividing these expressions,
\begin{equation*}
  \tan (n \theta) = \frac{\sum_{r=0}^{\lfloor(n-1)/2\rfloor} (-1)^{r}
    \binom{n}{2r+1} \tan^{2r+1}(\theta)}
  {\sum_{r=0}^{\lfloor n/2\rfloor} (-1)^{r} \binom{n}{2r} \tan^{2r}(\theta)}.
\end{equation*}
For $\theta = \arctan(x)$, this becomes
\[
\tan (n \arctan x) = 
\frac{\sum_{r=0}^{\lfloor(n-1)/2\rfloor} (-1)^{r} \binom{n}{2r+1} x^{2r+1}}
  {\sum_{r=0}^{\lfloor n/2\rfloor} (-1)^{r} \binom{n}{2r} x^{2r}},
\]
so this is an example of \eqref{eq:f(x)c} with $c=0$; namely a rational function. 

By convenience, let us denote
\begin{equation*}
  \numer_n(x) = \sum_{r=0}^{\lfloor(n-1)/2\rfloor} (-1)^{r} \binom{n}{2r+1} x^{2r+1},
  \quad
  \denom_n(x) = \sum_{r=0}^{\lfloor n/2\rfloor} (-1)^{r} \binom{n}{2r} x^{2r},
\end{equation*}
with $\numer_0 = 0$ and $\denom_0 = 1$.
Then, for $c=0$ in \eqref{eq:f(x)c}, we have
\begin{equation}
\label{eq:R0}
  R_0(n,x) = \frac{\numer_n(x)}{\denom_n(x)}, \quad n = 0,1,2,\dots.
\end{equation}
The above functions satisfy $\frac{d}{dx} \arctan(R_0(n,x)) = n/(1+x^2)$.
The first few of these functions are
\begin{gather*}
  R_0(0,x) = 0, 
  \quad
  R_0(1,x) = x, 
  \quad
  R_0(2,x) = \frac{-2x}{x^2-1}, 
  \quad
  R_0(3,x) = \frac{x^3-3x}{3x^2-1}, 
  \\
  R_0(4,x) = \frac{-4x^3+4x}{x^4-6x^2+1}, 
  \quad
  R_0(5,x) = \frac{x^5-10x^3+5x}{5x^4-10x^2+1}.
\end{gather*}

For the other values of $c$, we take $c = j\pi/4$ with $j=1,2,3$, and denote the corresponding solutions of \eqref{eq:f(x)c} by $R_1(n,x)$, $R_2(n,x)$ and $R_3(n,x)$, respectively.

For $c=\pi/2=2\pi/4$ it is clear that
\[
\tan(n\theta + \pi/2) = \frac{\sin(n\theta + \pi/2)}{\cos(n\theta + \pi/2)}
= \frac{-\cos(n\theta)}{\sin(n\theta)},
\]
so, for $c=\pi/2$, we have 
\begin{equation}
\label{eq:R2}
  R_2(n,x) = \frac{-\denom_n(x)}{\numer_n(x)} = \frac{-1}{R_0(n,x)},
  \quad n = 1,2,\dots.
\end{equation}
The above functions satisfy again $\frac{d}{dx} \arctan(R_2(n,x)) = n/(1+x^2)$.

For $c=\pi/4$,
\[
  \tan(n\theta + \pi/4) = \frac{\cos(n\theta)+\sin(n\theta)}{\cos(n\theta)-\sin(n\theta)},
\]
so, for $c=\pi/4$, we have 
\begin{equation}
\label{eq:R1}
  R_1(n,x) = \frac{\denom_n(x)+\numer_n(x)}{\denom_n(x)-\numer_n(x)},
  \quad n = 0,1,2,\dots.
\end{equation}
The above functions satisfy again $\frac{d}{dx} \arctan(R_1(n,x)) = n/(1+x^2)$.
For instance, $R_1(0,x)=1$,
\begin{gather*}
  R_1(1,x) = \frac{-x-1}{x-1}, 
  \quad
  R_1(2,x) = \frac{x^2-2x-1}{x^2+2x-1}, 
  \quad
  R_1(3,x) = \frac{-x^3-3x^2+3x+1}{x^3-3x^2+3x+1}, 
  \\
  R_1(4,x) = \frac{x^4-4x^3-6x^2+4x+1}{x^4+4x^3-6x^2-4x+1},
  \quad
  R_1(5,x) = \frac{-x^5-5x^4+10x^3+10x^2-5x-1}{x^5-5x^4-10x^3+10x^2+5x-1}.
\end{gather*}

Finally, for $c=3\pi/4$,
\[
\tan(n\theta + 3\pi/4) = \tan(n\theta - \pi/4) 
= \frac{\sin(n\theta)-\cos(n\theta)}{\sin(n\theta)+\cos(n\theta)},
\]
so, for $c=3\pi/4$, we have the functions
\begin{equation}
\label{eq:R3}
  R_3(n,x) = \frac{\numer_n(x)-\denom_n(x)}{\numer_n(x)+\denom_n(x)}
  = \frac{-1}{R_1(n,x)},
  \quad n = 0,1,2,\dots.
\end{equation}
Once more, they satisfy $\frac{d}{dx} \arctan(R_3(n,x)) = n/(1+x^2)$.

Actually, another way to define the functions $R_j(n,x)$, $j=0,1,2,3$, $n\in\N$, is to take
\begin{equation}
\label{eq:defR}
  R_j(n,x) = \tan(n\theta+j\pi/4),
  \qquad x = \tan\theta;
\end{equation}
this definition is valid in a small range of $x$ (to ensure that both the functions $\tan$ and $\arctan$ are invertible). Then the previous arguments show that these functions are rational functions, and, moreover, we have found their explicit expressions. Of course, once that we have a rational function defined in a small interval, we can extend it to the entire~$\C$.

\subsection{Some properties of the functions $R_j(n,x)$}

Here we present some of the algebraic properties of the functions $R_j(n,x)$.
Actually, some of these properties are not related to the Machine-like formulas, but they are interesting by themselves.

Let us first note that, because the function $\tan$ is odd, we could instead use the functions $-R_j(n,x)$ for our purposes; actually, we could use $R_j(-n,x)$ to denote them, because $\arctan(-R_j(n,x)) = -n/(1+x^2)$, a formula which holds by looking at~\eqref{eq:atanf}. This would allow to index the functions $R_j(n,x)$ over $n \in \Z$, but this fact does not have any practical contribution to finding additional Machin-like formulas.

When handling Machin-like formulas, it is more interesting to observe that the relation between $R_j(n,1/x)$ and $R_j(n,x)$, depends on whether $n$ is even or odd:
\begin{equation}
\label{eq:1/x}
\begin{gathered}
  R_0(2n,1/x) = - R_0(2n,x),
  \quad
  R_0(2n+1,1/x) = 1/R_0(2n+1,x),
  \qquad n = 0,1,2,\dots,
  \\
  R_1(2n,1/x) = 1/R_1(2n,x),
  \quad
  R_1(2n+1,1/x) = -R_1(2n+1,x),
  \qquad n = 0,1,2,\dots.
\end{gathered}
\end{equation}
The proofs of these properties are straightforward, so we do not include them. 
In relation to $R_j(n,-x)$, some symmetry properties also hold: 
\begin{equation}
\label{eq:-x}
  R_0(n,-x) = - R_0(n,x),
  \quad 
  R_1(n,-x) = 1/R_1(n,x),
  \qquad n = 0,1,2,\dots.
\end{equation}
Both for \eqref{eq:1/x} and for~\eqref{eq:-x}, the corresponding properties for $R_2$ and $R_3$ can be easily established from the properties of $R_0$ and $R_1$ using \eqref{eq:R2} and~\eqref{eq:R3}, respectively.

According to~\eqref{eq:compos}, the composition of the functions $R_j(n,x)$ generates new functions that are useful in relation to the Machin-like formulas. However, we can see that these functions are not really new. Let us start analyzing a particular case.

Let us first observe that, if the ``internal'' function in the composition is $R_0$, we have
\begin{align*}
R_j(nm,x) &= \tan(nm \arctan(x) + \pi j/4)
= \tan(n \arctan( \tan(m \arctan(x))) + \pi j/4) \\
&= \tan(n \arctan(R_0(m,x)) + \pi j/4)
= R_j(n,R_0(m,x)),
\qquad\qquad j=0,1,2,3;
\end{align*}
this argument is correct in a small enough interval of $x$'s (to guarantee that $\arctan \circ \tan = \operatorname{Id}$), and then by analytic continuation we can ensure that
\begin{equation}
\label{eq:Rjnm}
R_j(nm,x) = R_j(n,R_0(m,x)), \qquad x \in \C, \qquad j=0,1,2,3.
\end{equation}
For each $j,$ this formula allows to compute $R_j(n,x)$ as composition of successive $R_0(p_i,x)$ with a final $R_j(p_i,x)$, where the $p_i$ are the prime factors of~$n$. In particular, it is enough to know $R_j(p,x)$ for primes $p$ in order to generate (or to compute) all the $R_j(n,x)$ by composition.

With full generality, it is not difficult to check that the composition of functions $R_j$ behaves as follows:
\begin{equation*}
  R_j(n,R_i(m,x)) = R_{in+j \operatorname{mod} 4}(nm,x).
\end{equation*}

Let us note that the relation $R_0(n,R_0(m,x)) = R_{0}(nm,x)$ of the functions $R_0$ coincides with the property $T_n(T_m(x)) = T_{nm}(x)$ satisfied by the Chebychev polynomials of the first kind $T_n(x) := \cos(n\arccos (x))$, $x \in [-1,1]$. These were used in \cite{BeNaVa, CiNaVa} in relation to the M\"obius inversion formula. Finally, let us also mention that, although with a different notation, the functions $R_0(n,x)$ have been already defined in~\cite{Cal1} (in particular, their rational expressions are given), but they have not been used to obtain Machin-like identities. As we will comment a little later, the functions $R_3(n,x)$ have been already introduced in~\cite{ABCM} with a different approach.

\subsection{Machin-like formulas associated to $R_j(n,x)$}
\label{subsec:ss}

Once we have defined the functions $R_j(n,x)$ and studied their properties, we can state the main result of this section. 

Before doing that, let us observe the following:
\begin{itemize}
\item[(a)] At $x=0$ or $x=\pm\infty$, the value of the function $\arctan(R_j(n,x))$ (perhaps in the sense of a limit) is always a rational multiple of~$\pi$.
\item[(b)] The functions $\arctan(R_j(n,x))$ are not defined at the roots of the denominator of $R_j(n,x)$. However, and because $\arctan(\infty) - \arctan(-\infty) = \pi$, it is clear that, at every $x$ that is a root of the denominator of $R_j(n,x)$, the jump $\arctan(R_j(n,x^+))-\arctan(R_j(n,x^-))$ is a multiple of~$\pi$.
\end{itemize}

Let us now take any function of the form
\begin{equation}
\label{eq:Fmachin}
  F(x) = \sum_{k=1}^{N} \frac{r_k}{n_k} \arctan(R_{j_k}(n_k,x))
  \qquad\text{with}\quad \sum_{k=1}^{N} r_k = 0,
\end{equation}
defined in $\R$ except at the roots of the denominators.
Since $\frac{d}{dx} \arctan(R_{j}(n,x)) = n/(1+x^2)$, it is clear that
\[
  F'(x) = \sum_{k=1}^{N} \frac{r_k}{n_k} \frac{n_k}{1+x^2} = 0,
\]
so the function $F(x)$ is piecewise constant (the continuity and the differentiability disappear only at the roots of the denominators). Thus, as a consequence of \eqref{eq:Fmachin}, (a) and (b), we have the following result.

\begin{theorem}
\label{theo:machinRj}
Let $R_0(n,x)$, $R_1(n,x)$, $R_2(n,x)$ and $R_3(n,x)$ be the rational functions with integer coefficients defined in \eqref{eq:R0}, \eqref{eq:R1}, \eqref{eq:R2} and \eqref{eq:R3}, respectively, with $n=0,1,2,\dots$, and let $r_k$, $k=1,2,\dots,N$, be integers such that $\sum_{k=1}^{N} r_k = 0$. 
Then, for any $x \in \Q$ we have the Machin-like formula
\begin{equation}
\label{eq:xmachinRj}
  \sum_{k=1}^{N} \frac{r_k}{n_k} \arctan(R_{j_k}(n_k,x)) = \frac{r}{s} \pi,
\end{equation}
with $r/s \in \Q$
(notice that, as the $R_{j}(n,x)$ are rational functions with integer coefficients, the functions $\arctan$ that appear in~\eqref{eq:xmachinRj} are evaluated at rational values).
\end{theorem}

Let us make some comments on this result.
First observe that $r/s$ is constant on intervals of the variable $x$, but the constant changes when any of the $R_{j}(n,x)$ involved in~\eqref{eq:xmachinRj} has a root at the denominator. 
Figures~\ref{fig:ej1}, \ref{fig:ej2} and~\ref{fig:ej3} show, in a graphical way, three examples of the theorem (they are simple examples, without any special interest).
Observe that, with the notation of Theorem~\ref{theo:machinRj}, the coefficients of $\arctan$ in Figure~\ref{fig:ej1} should be written as $\frac{12}{3}$ and $\frac{-12}{4}$, respectively, to get $r_1+r_2=0$; and the same in Figure~\ref{fig:ej2} with $\frac{91}{13}$ and $\frac{-91}{7}$.

Actually, it can happen that $F(x)$ in \eqref{eq:Fmachin} (or the left-hand-side sum in~\eqref{eq:xmachinRj}) is the constant zero function in some of those intervals, and then $r/s = 0$ in that interval;
but, of course, this is not the usual situation. 
For instance, this happens around $x=-1$ in the case of Figure~\ref{fig:ej1}.

\begin{figure}[p] 
\centering
\includegraphics[width=0.6\textwidth]{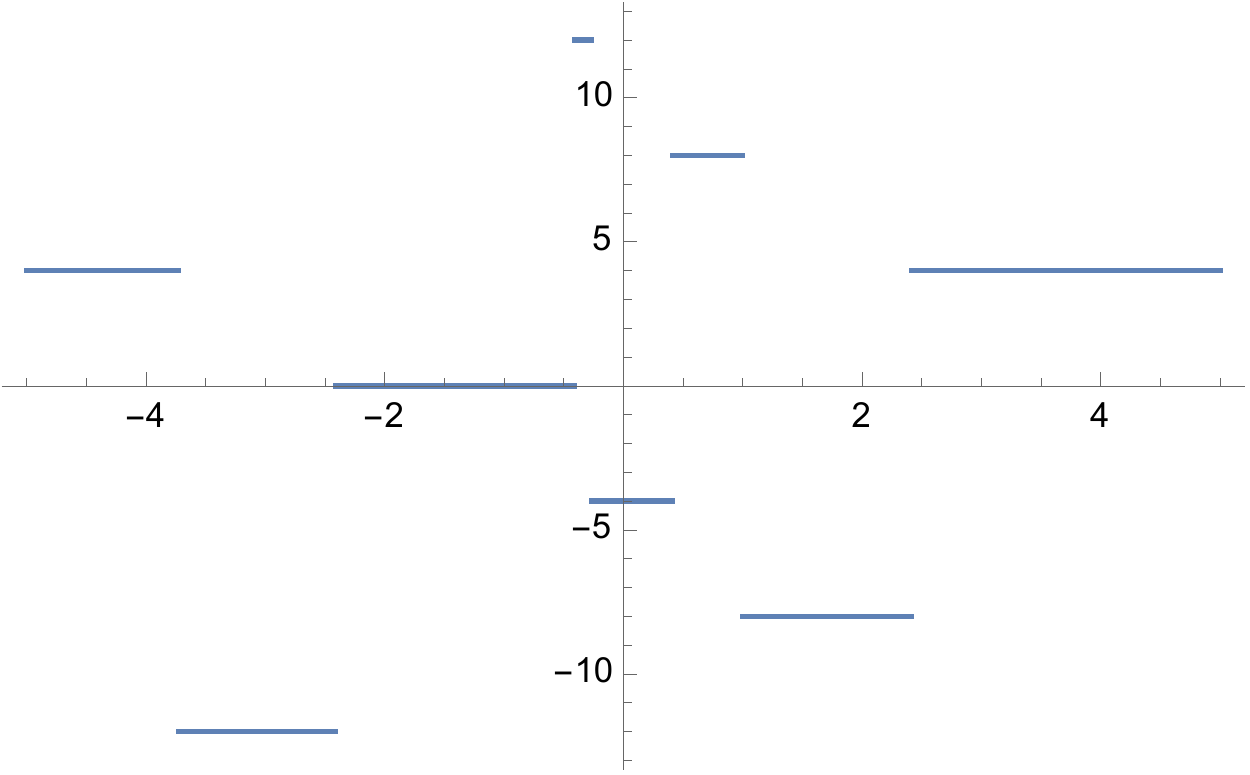}
\caption{The function 
$\frac{4}{\pi} \big(4\arctan(R_3(3,x)) - 3\arctan(R_0(4,x))\big)$, 
for $x \in (-5,5)$.}
\label{fig:ej1}
\end{figure}

\begin{figure}[p]
\centering
\includegraphics[width=0.6\textwidth]{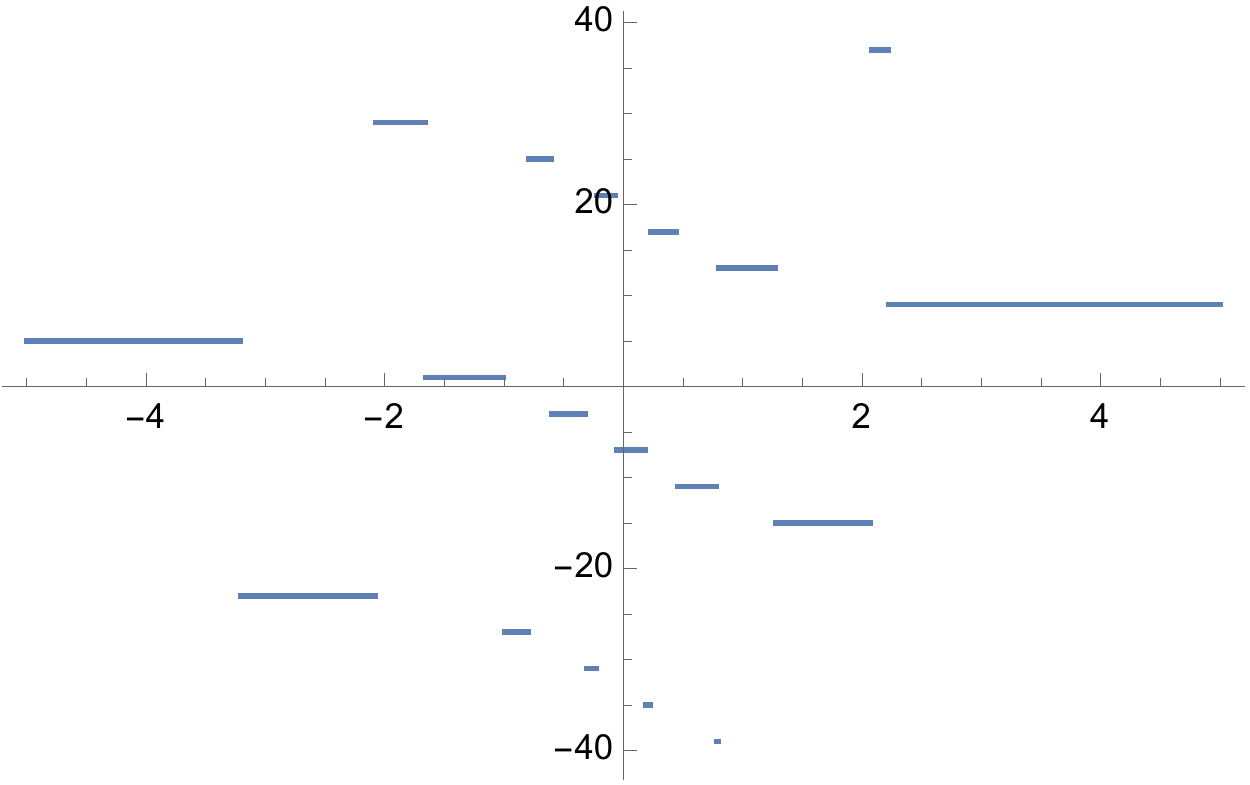}
\caption{The function 
$\frac{4}{\pi} \big(7\arctan(R_3(13,x)) - 13\arctan(R_0(7,x))\big)$, 
for $x \in (-5,5)$.}
\label{fig:ej2}
\end{figure}

\begin{figure}[p]
\centering
\includegraphics[width=0.6\textwidth]{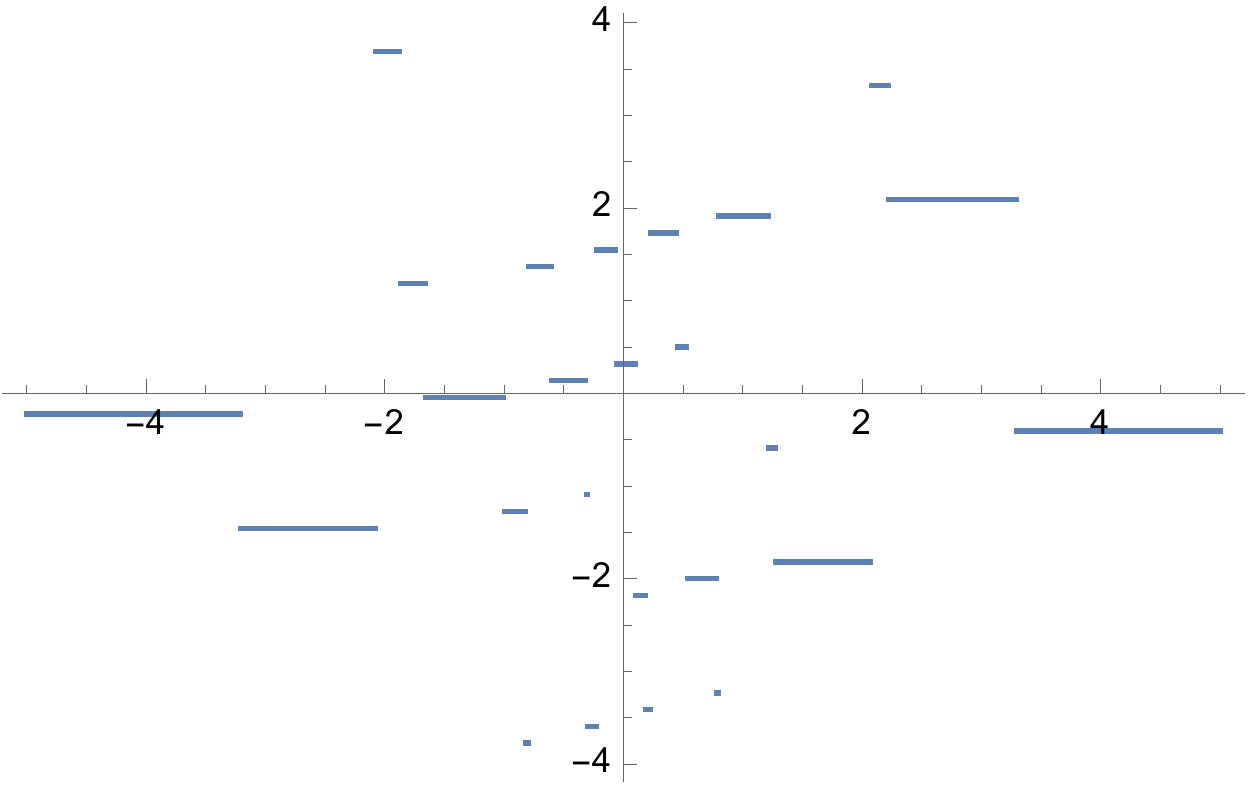}
\caption{The function 
$\frac{4}{\pi} \big(\frac{4}{13}\arctan(R_3(13,x)) - \frac{9}{7}\arctan(R_0(7,x)) + \frac{5}{8}\arctan(R_1(8,x))\big)$, 
for $x \in (-5,5)$.}
\label{fig:ej3}
\end{figure}

An important point is that, if we want to use \eqref{eq:xmachinRj} to evaluate $\pi$ using the Taylor expansion~\eqref{eq:Gregory}, we need that the $R_{j_k}(n_k,x)$ satisfy $|R_{j_k}(n_k,x)| < 1$. Let us now recall that
\begin{equation}
\label{eq:jj'}
R_2(n,x) = \frac{-1}{R_0(n,x)} 
\qquad\text{and}\qquad
R_3(n,x) = \frac{-1}{R_1(n,x)}.
\end{equation}
Moreover, the function $\arctan$ satisfies $\arctan(-1/t) = -\arctan(1/t)$ and
\begin{equation}
\label{eq:atant}
 \arctan\left(\frac{1}{t}\right) = \sgn(t) \frac{\pi}{2} - \arctan(t).
\end{equation}
Thus, if we use instead $R_{j'}$ with the notation $0'=2$, $1'=3$, $2'=0$ and $3'=1$, in the case of \eqref{eq:xmachinRj} with a $|R_{j_k}(n_k,x)| > 1$ we can replace $\arctan(R_{j_k}(n_k,x))$ by the corresponding $\arctan(R_{j'_k}(n_k,x))$, that will satisfy $|R_{j'_k}(n_k,x)| < 1$, so we can use the Taylor expansion. (This cannot be done if $R_{j_k}(n_k,x) = \pm1$, but to evaluate our expression in those $x$ is of no interest because $\arctan(\pm1) = \pm \pi/4$, so the corresponding summand $\arctan(R_{j_k}(n_k,x))$ can be removed from the formula.)

The use of \eqref{eq:atant} in the above mentioned procedure that replaces $R_{j_k}$ by $R_{j'_k}$ modifies the identity \eqref{eq:xmachinRj} to a new identity of the same kind with all the $|R_{j}(n,x)| < 1$. In this process, and since $\arctan(-1/t) = -\arctan(1/t)$, the corresponding $r_k$ in the condition $\sum_{k=1}^{N} r_k = 0$ on the coefficients becomes $-r_k$. But, at the same time, the value of $r/s$ changes; in particular, it can become to be $0$ and in this case we get a useless Machin-like identity.

Finally, we want to comment that Theorem~\ref{theo:machinRj} extends some of the results of \cite{ABCM}, where the authors prove that, for a positive integer~$n$,
\[
n \arctan\left(\frac{1}{x}\right) + \arctan\left(R_3(n,x)\right) = \frac{r}{s} \pi,
\]
with $r/s \in \Q$. Recall that $R_2(1,x)=-1/x$ and then this equality can also be written as
\[
-\frac{n}{1} \arctan\left(R_2(1,x)\right) + \frac{n}{n} \arctan\left(R_3(n,x)\right) = \frac{r}{s} \pi,
\]
that is of the form~\eqref{eq:xmachinRj}.
In that paper, the rational functions $R_3(n,x)$ are obtained in a very different way. They are given via a recurrent relation between polynomials that are a particular case of the so called R\'edei polynomials, see~\cite{Redei}.

\subsection{Some examples}

With the notation of the $R_j(n,x)$, the ten sporadic cases of Theorem~\ref{theo:power2} (in particular, this includes the four classical examples by Machin, Euler, Hermann and Hutton mentioned in the introduction) can be obtained, in the same order as in the theorem, as follows:
\begin{gather*}
4 \arctan(R_0(1, x)) - \arctan(R_3(4, x)) = \pi/4, 
\quad\text{with } x = 1/5,
\\
2 \arctan(R_0(1, x)) - \arctan(R_3(2, x)) = \pi/4, 
\quad\text{with } x = 1/2,
\\
-\tfrac{3}{2} \arctan(R_0(2, x)) + \arctan(R_3(3, x)) = \pi/4, 
\quad\text{with } x = 3,
\\
- \arctan(R_0(3, x)) + 3 \arctan(R_3(1, x)) = \pi/4, 
\quad\text{with } x = 2,
\\
\tfrac{4}{3} \arctan(R_0(1, x)) - \tfrac{1}{3} \arctan(R_1(4, x)) = \pi/4, 
\quad\text{with } x = 2/3,
\\
\tfrac{1}{2} \arctan(R_0(3, x)) - \tfrac{3}{2} \arctan(R_2(1, x)) = \pi/4, 
\quad\text{with } x = 2,
\\
2 \arctan(R_0(1, x)) - \arctan(R_3(2, x)) = \pi/4, 
\quad\text{with } x = 2/5,
\\
 2 \arctan(R_0(1, x)) - \arctan(R_3(2, x)) = \pi/4, 
\quad\text{with } x = 1/3, 
\\
-\tfrac{1}{2} \arctan(R_0(2, x)) + \arctan(R_3(1, x)) = \pi/4, 
\quad\text{with } x = 3,
\\
2 \arctan(R_0(3, x)) - 3 \arctan(R_1(2, x)) = \pi/4, 
\quad\text{with } x = 2,
\end{gather*}
while the two parametric families correspond to 
\begin{gather*}
\arctan \left(R_0\left(1, \frac{1}{2^{a+1}+1}\right)\right)
- \arctan \left(R_3\left(1, \frac{1}{2^{a+1}+1}\right)\right) = \frac{\pi}{4}, 
\\
\arctan \left(R_0\left(1, \frac{1}{2^{a+1}-1}\right)\right)
- \arctan \left(R_3\left(1, \frac{1}{2^{a+1}-1}\right)\right) = \frac{\pi}{4},
\end{gather*}
for $a\in \N^*$. 

It is not difficult to identify many other two-term well-known Machin-like formulas by means of our notation. Let us give some examples, with their corresponding Lehmer measures, denoted by~$\mu$.
The combination
$5\arctan(R_1(2,x)) - 2\arctan(R_0(5,x))$ for $x = 3$ gives the formula 
\begin{equation}
\label{eq:GI}
5 \arctan\left(\frac{1}{7}\right) + 2 \arctan\left(\frac{3}{79}\right) = \frac{\pi}{4},
\qquad \mu \sim 1.88727.
\end{equation}
The combination
$22\arctan(R_2(17,x)) - 17\arctan(R_3(22,x))$ for $x = 1/2$ gives
\begin{equation}
\label{eq:GI2}
22 \arctan\left(\frac{24\,478}{873\,121}\right) 
+ 17 \arctan\left(\frac{685\,601}{69\,049\,993}\right) 
= \frac{\pi}{4},
\qquad \mu \sim 1.14343.
\end{equation}
Finally,
$22\arctan(R_0(1,x)) - \arctan(R_3(22,x))$ for $x = 1/28$ gives 
\[
  22 \arctan\left(\frac{1}{28}\right)
  + \arctan\left(\frac{1\,744\,507\,482\,180\,328\,366\,854\,565\,127}
  {98\,646\,395\,734\,210\,062\,276\,153\,190\,241\,239}\right) = \frac{\pi}{4},
\qquad \mu \sim 0.901429.
\]
Of course, each of the Machin's formulas appearing in this paper can be checked by direct multiplication of its associated Gaussian integers. For instance, \eqref{eq:GI} and \eqref{eq:GI2} hold because
$(7+i)^5(79+3i)^2 = 2^35^{10}(1+i)$ and
\[
  (873\,121 + 24\,478i)^{22} (69\,049\,993+685\,601i)^{17} = 2^8 5^{374}(1+i).
\]

It seems to us that our formulas are likely to reproduce most of the known Machin's type formulas with two terms, as well as to obtain new ones with $N=2$, but are not enough in general to include all formulas with $N>2$, like for instance the ones appearing in~\cite{Gu}. In any case, two term formulas have been also shown to be useful as starting points to produce formulas with more terms, see for instance the procedures developed in~\cite{AbSJQ2, ChaHe, Todd, We}.

Although the first main aim of our paper was to produce Machin-like identities with arbitrarily small Lehmer measure, with the help of Theorem~\ref{theo:machinRj}, it is not difficult to look for examples satisfying this property. It is enough to take $N=2$ and to use a suitable strategy, with the help of any computer algebra system.

We want to get two functions $R_{j}(n,x)$ and $R_{i}(m,x)$ whose absolute values are ``small'' at the same $x$, to guarantee that the corresponding series \eqref{eq:Gregory} converges quickly (that is, ``few'' summands of the series are necessary to get a good precision). This is what happens with the Machin-like identities having small Lehmer measure. With the aid of a computer, we can look for these $x$ with different strategies: 
\begin{itemize}
\item Searching numerically for minima of each function of type $R_j(n,x)^2 + R_i(m,x)^2$ (or similar, since, for example, we can put different weights or exponents on the two summands), and imposing that the value of the resulting function is small enough.
\item By numerically identifying intervals in which, simultaneously, $-\varepsilon < R_j(n,x) < \varepsilon$ and $-\varepsilon < R_i(m,x) < \varepsilon$, for $\varepsilon>0$ fixed beforehand.
\end{itemize}
In both cases, in order to obtain ``nice'' expressions, it is of interest to take $x$ rational with a numerator and denominator that are not too large. This can be achieved by taking convergents of continued fractions of numbers that we have obtained with the previous strategies.

A couple of new Machin-Like identities have been obtained with the above strategy (with their corresponding Lehmer measure~$\mu$). Using $R_j(n,x)$ with big values of $n$ it is easier to find examples with small Lehmer measure, but then we end up with fractions $a_k/b_k$ where both $a_k$ and $b_k$ have many digits. In this case, we denote $f^{r}_{s}$ to indicate an irreducible fraction with $r$ digits in the numerator and $s$ digits in the denominator.
\begin{itemize}
\item The relation $33\arctan(R_0(1, x)) - \arctan(R_3(33, x))$ with $x=1/42$ gives
\[
  33 \arctan(1/42) - \arctan(\ff{50}{54}) = \pi/4, 
  \quad \mu \sim 0.880916.
\]
\item The relation $48\arctan(R_0(1,x)) - \arctan(R_3(48,x))$ with $x=9/550$ gives
\[
  48 \arctan(9/550) - \arctan(\ff{127}{132}) = \pi/4, 
  \quad \mu \sim 0.765513.
\]
\end{itemize}

To evaluate $R_j(n,x)$ for very big values of $n$ (say, for instance, $n > 100$), it is not a good idea to use their rational expressions given in~\eqref{eq:R0}, \eqref{eq:R2}, \eqref{eq:R1} and~\eqref{eq:R3}. 
For instance when $n$ is odd, both the numerator and the denominator are polynomials with $n+1$ non-zero monomials, see~\eqref{eq:R1}. 
From a computational point of view, it is better to proceed as follows. In practice, we have used these methods in some of the examples that appear in the next section.

Because $R_0(n,x) = \tan(n\theta)$ with $x = \tan \theta$
we have 
\[
  R_0(n,x) = \tan(n\theta) = \frac{\sin(n\theta)}{\cos(n\theta)}
  = \frac{\Im\big((\cos\theta + i\sin\theta)^n\big)}
    {\Re\big((\cos\theta + i\sin\theta)^n\big)}.
\]
Dividing both the numerator and the denominator by $\cos^n(\theta)$ we get
\[
  R_0(n,x) 
  = \frac{\Im\big((\cos\theta + i\sin\theta)^n\big)}
    {\Re\big((\cos\theta + i\sin\theta)^n\big)}
  = \frac{\Im\big((1+i\tan\theta)^n\big)}
    {\Re\big((1+i\tan\theta)^n\big)}
  = \frac{\Im\big((1+ix)^n\big)}
    {\Re\big((1+ix)^n\big)}.
\]
If $x \in \Q$ is fixed, we can evaluate $(1+ix)^n$ via successive squaring (this is particularly easy if $n$ is a power of~$2$). 
Thus, $(1+ix)^n$ is a number in $\Q[i]$, and using it we obtain $R_0(n,x)$.

In the same way, using $R_1(n,x) = \tan(n\theta+\pi/4)$ with $x = \tan \theta$, we have
\begin{align*}
  R_1(n,x) &= \frac{\sin(n\theta+\pi/4)}{\cos(n\theta+\pi/4)}
  = \frac{\cos(n\theta)+\sin(n\theta)}{\cos(n\theta)-\sin(n\theta)} \\
  &= \frac{\Re\big((\cos\theta + i\sin\theta)^n\big) 
      + \Im\big((\cos\theta + i\sin\theta)^n\big)}
    {\Re\big((\cos\theta + i\sin\theta)^n\big)
      - \Im\big((\cos\theta + i\sin\theta)^n\big)} \\
  &= \frac{\Re\big((1 + i\tan\theta)^n\big) 
      + \Im\big((1 + i\tan\theta)^n\big)}
    {\Re\big((1 + i\tan\theta)^n\big)
      - \Im\big((1 + i\tan\theta)^n\big)}
  = \frac{\Re\big((1+ix)^n\big) + \Im\big((1+ix)^n\big)}
    {\Re\big((1+ix)^n\big) - \Im\big((1+ix)^n\big)},
\end{align*}
and again we can evaluate $(1+ix)^n$ by means of successive squaring.
Using \eqref{eq:jj'}, we get the corresponding expressions for $R_2(n,x)$ and $R_3(n,x)$.

In the particular case of $n = 2^m$, there is another clever way to evaluate $R_j(2^m,x)$:
using~\eqref{eq:Rjnm} we can write $R_j(2^m,x)$ as a composition of successive $R_0(2,x)$ with a final $R_j(2,x)$; i.e.,
\begin{equation*}
  R_j(2^m,x) = R_j(2,R_0^{\circ(m-1)}(2,x))
\end{equation*}
(to avoid confusion with multiplicative powers, we use $f^{\circ n}$ to denote the composition of the function $f$ with itself $n$ times).
In the above, we have $m$ rational functions (with numerators and denominators of degree $1$ or~$2$) that are easy to evaluate. 
Computer experiments show that, for $n=2^m$, this method is faster than the previous procedure based on computing $(1+ix)^n$ via successive squaring. 

\subsection{Machin-like identities with small Lehmer measure}

With the help of the functions $R_j(n,x)$, we can prove that there exist two-term Machin-like identities with Lehmer measure as small as we want. To do this, we use standard properties of the continued fractions. The formulas with Lehmer measure as small as desired can be given explicitly.

For a real number $x$, let us denote its continued fraction by $x = [c_0,c_1,c_2,c_3,\dots]$, and let $p_k/q_k = [c_0,c_1,c_2,\dots,c_k]$, with $k=0,1,2,\dots$, be its convergents. 
It is well known that
\begin{equation}
\label{eq:cotapkqk}
  \left|x - \frac{p_k}{q_k}\right| \le \frac{1}{q_k^2}.
\end{equation}

\begin{theorem}
\label{theo:mLeps}
For every $\varepsilon > 0$ there exist positive integers $n,b_1,b_2$, and another integer $a_2$ with $0<|a_2|<b_2$, such that the Machin-like identity
\begin{equation}
\label{eq:mLeps}
  n \arctan\frac{1}{b_1} - \arctan\frac{a_2}{b_2} = \frac{\pi}{4}
\end{equation}
has Lehmer measure \eqref{eq:Lehmer-orig} less than~$\varepsilon$.
\end{theorem}

\begin{proof}
Let $p_k/q_k$ be the convergents of the continued fraction of~$\pi$. By~\eqref{eq:cotapkqk},
\[
  \left|\pi - \frac{p_k}{q_k}\right| = O\left(\frac{1}{q_k^2}\right),
\]
so
\[
  \left|\frac{\pi}{4p_k} - \frac{1}{4q_k}\right| = O\left(\frac{1}{q_k^3}\right).
\] 
Taking $\xi=1/(4q_k)$, the alternating series \eqref{eq:Gregory} easily gives
\[
  \arctan\left(\frac{1}{4q_k}\right) = \frac{1}{4q_k} + O\left(\frac{1}{q_k^3}\right) 
  = \frac{\pi}{4p_k} + O\left(\frac{1}{q_k^3}\right).
\]
Multiplying by $n=p_k$, we obtain
\begin{equation}
\label{eq:pi/4+O}
  n\arctan(\xi) = p_k \left(\frac{\pi}{4p_k} + O\left(\frac{1}{q_k^3}\right)\right) 
  = \frac{\pi}{4} + O\left(\frac{1}{q_k^2}\right). 
\end{equation}

Note that the derivative of $g(x) = n\arctan(x) - \arctan(R_3(n,x))$ is $0$, so $g$ is piecewise constant. Because $g(0) = n \arctan(0) - \arctan(R_3(n,0)) = 0 - \arctan(-1) = \pi/4$, there exists an interval $\mathcal{I}$ around $0$ where the value the function is $\pi/4$ and on~$\mathcal{I}$,
\begin{equation}
\label{eq:ctepi/4}
n\arctan(x) - \arctan(R_3(n,x)) = \frac{\pi}{4}.
\end{equation}
It suffices to show that $\xi=1/(4q_k)$ belongs to~$\mathcal{I}$, which we do below.

By definition, see~\eqref{eq:defR}, the formula
\[
  R_3(n,x) = \tan\left(n\theta+\frac{3\pi}{4}\right) = \tan\left(n\theta-\frac{\pi}{4}\right),
  \qquad x = \tan\theta,
\]
is valid on an interval for the variable $\theta$ on which $\tan$ is a bijective function. That is, for $-\pi/2 < n\theta-\pi/4 < \pi/2$, or $-\pi/4 < n\theta < 3\pi/4$. Because $\tan\theta \sim \theta$ for small $\theta$, this is the case when $-\pi/5 < nx < 3\pi/5$ and $\theta$ is close to zero. Under this condition, we have
\[
  n\arctan(x) - \arctan(R_3(n,x))
  = n\theta - \arctan\left(\tan\left(n\theta-\frac{\pi}{4}\right)\right)
  = n\theta - \left(n\theta-\frac{\pi}{4}\right) = \frac{\pi}{4},
\]
as desired. Since certainly, $\xi=1/(4q_k)$ satisfies $-\pi/5 < n\xi = p_k/(4q_k) < 3\pi/5$, because $p_k/q_k$ are the convergents of~$\pi$.

Once proved that $\xi=1/(4q_k)$ satisfies~\eqref{eq:ctepi/4}, it follows from \eqref{eq:pi/4+O} and \eqref{eq:ctepi/4} that $|R_3(n,\xi)| = O(1/q_k^2)$.
We then have the Machin-like formula~\eqref{eq:mLeps} with $b_1 = 1/\xi = 4q_k$ and $a_2/b_2 = R_3(n,\xi)$.

Finally,
\[
\frac{1}{\log_{10}(1/\xi)} + \frac{1}{\log_{10}(1/|R_3(n,\xi)|)} 
= O\left(\frac{1}{\log_{10} q_k}\right) \stackrel{(*)}{=} O\left(\frac{1}{k}\right),
\]
so taking $k$ big enough, the thesis follows. The step $(*)$ can be justified as follows: the recurrence relation
\[
q_k = a_{k} q_{k-1}+q_{k-2} \ge q_{k-1}+q_{k-2} 
\]
(with $a_k\ge 1$ being the partial quotients of the continued fraction) gives that $q_k\ge F_k$, where $F_m$ is the $m$-th Fibonacci number. Consequently, $q_k\ge \phi^{k-2}$ with $\phi$ the golden section, so $\log_{10} q_k\gg k$.
\end{proof}

\begin{corollary} 
For every $\varepsilon > 0$ and every $N\ge2$ there exists a Machin-like identity
\begin{equation*}
  \sum_{k=1}^{N} \frac{r_k}{n_k} \arctan\left(\frac{a_k}{b_k}\right) = \frac{\pi}{4}
  \qquad \text{with}\quad \prod_{k=1}^{N} r_k\ne 0, 
\end{equation*}
which has Lehmer measure \eqref{eq:Lehmer-orig} less than~$\varepsilon$.
\end{corollary}
 
\begin{proof}
Given one of the two terms formulas obtained in Theorem~\ref{theo:mLeps}, with arbitrarily small Lehmer measure, any of its arctangent terms can be split into two new ones by using the well known identity
\[
  \arctan(x) = 2\arctan(2x)-\arctan(4x^3+3x),
\]
which once more can be easily proved by derivation. By applying this procedure $N-2$ times we arrive to the desired result. As we have already commented, other ways to split one arctangent term into several ones are developed in~\cite{AbSJQ2, ChaHe, Todd, We}.
\end{proof}

Note that the above proof is constructive, and we can use the procedure given in the proof to explicitly state Machin-like formulas, see Table~\ref{tab:Machineps}. We used the successive squaring method explained in the previous section to compute the values $a_2/b_2$ that appear in that table. 

\begin{table}
\footnotesize
\newcommand{\vv}{\rule[-4pt]{0pt}{14pt}\ignorespaces}
\centering
\begin{tabular}{ccccc}
\hline 
\vv $k$ & $p_k/q_k$ 
& $a_1/b_1$ & $a_2/b_2$ & Lehmer measure
\\
\hline
\vv $1$ & $22/7$
& $1/28$ & $\ff{28}{32} \sim \np{0.0000176845}$ & $\np{0.901429}$ \\
\vv $2$ & $333/106$
& $1/424$ & $\ff{871}{876} \sim \np{0.0000222611}$ & $\np{0.59555}$ \\
\vv $3$ & $355/113$
& $1/452$ & $\ff{937}{943} \sim \np{1.21473}\cdot10^{-6}$ & $\np{0.545675}$ \\
\vv $4$ & $\np{103993}/\np{33102}$
& $1/\np{132408}$ & $\ff{532634}{532644} \sim \np{1.59405}\cdot10^{-10}$ & $\np{0.297306}$ \\
\vv $5$ & $\np{104348}/\np{33215}$
& $1/\np{132860}$ & $\ff{534606}{534617} \sim \np{-6.80756}\cdot10^{-11}$ & $\np{0.29354}$ \\
\vv $6$ & $\np{208341}/\np{66317}$
& $1/\np{265268}$ & $\ff{1129966}{1129977} \sim \np{3.43096}\cdot10^{-11}$ & $\np{0.279937}$ \\
\vv $7$ & $\np{312689}/\np{99532}$
& $1/\np{398128}$ & $\ff{1751055}{1751066} \sim \np{-5.63418}\cdot10^{-12}$ & $\np{0.267466}$ \\
\vv $8$ & $\np{833719}/\np{265381}$
& $1/\np{1061524}$ & $\ff{5023921}{5023933} \sim \np{2.4112}\cdot10^{-12}$ & $\np{0.252025}$ \\
\vv $9$ & $\np{1146408}/\np{364913}$
& $1/\np{1459652}$ & $\ff{7066733}{7066745} \sim \np{-2.79808}\cdot10^{-13}$ & $\np{0.241887}$ \\
\vv $10$ & $\np{4272943}/\np{1360120}$
& $1/\np{5440480}$ & $\ff{28780982}{28780995} \sim \np{1.09862}\cdot10^{-13}$ & $\np{0.22563}$ \\
\vv $11$ & $\np{5419351}/\np{1725033}$
& $1/\np{6900132}$ & $\ff{37062153}{37062169} \sim \np{-3.75733}\cdot10^{-17}$ & $\np{0.207106}$ \\
\vv $12$ & $\np{80143857}/\np{25510582}$
& $1/\np{102042328}$ & $\ff{641854533}{641854548} \sim \np{1.69914}\cdot10^{-16}$ & $\np{0.188275}$ \\
\vv $13$ & $\np{165707065}/\np{52746197}$
& $1/\np{210984788}$ & $\ff{1379387210}{1379387226} \sim \np{-3.51397}\cdot10^{-17}$ & $\np{0.180906}$ \\
\vv $14$ & $\np{245850922}/\np{78256779}$
& $1/\np{313027116}$ & $\ff{2088646642}{2088646658} \sim \np{2.22166}\cdot10^{-17}$ & $\np{0.177756}$ \\
\vv $15$ & $\np{411557987}/\np{131002976}$
& $1/\np{524011904}$ & $\ff{3588514476}{3588514494} \sim \np{-3.88753}\cdot10^{-19}$ & $\np{0.172125}$ \\
\hline
\end{tabular}
\caption{Machin-like identities
$n\arctan(1/b_1) - \arctan(a_2/b_2) = \frac{\pi}{4}$,
with the notation of the proof of Theorem~\ref{theo:mLeps}. 
The notation $f^{r}_{s}$ is used to indicate an irreducible fraction with $r$ digits in the numerator and $s$ digits in the denominator.}
\label{tab:Machineps}
\end{table}

To conclude this section, let us see another way to obtain two-term Machin-like identities with small Lehmer measure. As shown in the proof of Theorem~\ref{theo:mLeps}, the function $g(x) = n\arctan(x) - \arctan(R_3(n,x))$ is piecewise constant and its value is $\pi/4$ in an interval around~$0$. 

\begin{table}
\footnotesize
\newcommand{\vv}{\rule[-4pt]{0pt}{14pt}\ignorespaces}
\newcommand{\igual}{\vv $=$}
\centering
\begin{tabular}{cccc}
\hline 
\vv $2^m$ & 
$a_1/b_1$ & $a_2/b_2$ & Lehmer measure
\\
\hline
\vv $2^5$ & 
$1/40$ & $\ff{50}{52} \sim \np{0.014436}$ & $\np{1.16751}$ \\
\igual 
& $1/41$ & $\ff{45}{47} \sim \np{-0.00506511}$ & $\np{1.0557}$ \\
\igual 
& $3/122$ & $\ff{65}{67} \sim \np{0.00132854}$ & $\np{0.969041}$ 
\\
\vv $2^6$ & 
$1/81$ & $\ff{111}{113} \sim \np{0.00468519}$ & $\np{0.953294}$ \\
\igual & 
$2/163$ & $\ff{138}{142} \sim \np{-0.000161494}$ & $\np{0.786967}$ \\
\igual 
& $39/\np{3178}$ & $\ff{220}{225} \sim \np{-0.0000379642}$ & $\np{0.749474}$ 
\\
\vv $2^7$ & 
$1/162$ & $\ff{281}{283} \sim \np{0.00471529}$ & $\np{0.88242}$ \\
\igual & 
$1/163$ & $\ff{261}{265} \sim \np{-0.000131942}$ & $\np{0.709799}$ \\
\igual & 
$39/\np{6356}$ & $\ff{482}{487} \sim \np{-8.39746}\cdot10^{-6}$ & $\np{0.649066}$ 
\\
\vv $2^8$ & 
$1/325$ & $\ff{603}{605} \sim \np{0.00229166}$ & $\np{0.776917}$ \\
\igual & 
$1/326$ & $\ff{640}{644} \sim \np{-0.000124553}$ & $\np{0.654001}$ \\
\igual & 
$19/\np{6193}$ & $\ff{927}{933} \sim \np{2.24663}\cdot10^{-6}$ & $\np{0.574947}$ 
\\
\vv $2^9$ & 
$1/651$ & $\ff{1361}{1364} \sim \np{0.00108355}$ & $\np{0.69267}$ \\
\igual & 
$1/652$ & $\ff{1438}{1442} \sim \np{-0.000122706}$ & $\np{0.611015}$ \\
\igual & 
$9/\np{5867}$ & $\ff{1848}{1853} \sim \np{0.0000111404}$ & $\np{0.557238}$ 
\\
\vv $2^{10}$ & 
$1/\np{1303}$ & $\ff{3033}{3036} \sim \np{0.000480424}$ & $\np{0.622385}$ \\
\igual & 
$1/\np{1304}$ & $\ff{3187}{3191} \sim \np{0.000122244}$ & $\np{0.576572}$ \\
\igual & 
$4/\np{5215}$ & $\ff{3803}{3807} \sim \np{0.0000283365}$ & $\np{0.540901}$ 
\\
\vv $2^{20}$ & 
$1/\np{1335088}$ & $\ff{6423057}{6423063} \sim \np{2.52287}\cdot10^{-7}$ & $\np{0.31481}$ \\
\igual & $2/\np{2670177}$ & $\ff{6738709}{6738716} \sim \np{-4.18498}\cdot10^{-8}$ & $\np{0.298784}$ \\
\igual & $7/\np{9345619}$ & $\ff{7151377}{7151387} \sim \np{1.6974}\cdot10^{-10}$ & $\np{0.265604}$ 
\\
\vv $2^{21}$ & 
$1/\np{2670176}$ & $\ff{13477425}{13477432} \sim \np{2.52287}\cdot10^{-7}$ & $\np{0.307163}$ \\
\igual & $1/\np{2670177}$ & $\ff{13161772}{13161779} \sim \np{-4.18497}\cdot10^{-8}$ & $\np{0.291137}$ \\
\igual & $7/\np{18691238}$ & $\ff{15249721}{15249731} \sim \np{1.69851}\cdot10^{-10}$ & $\np{0.25796}$ 
\\
\vv $2^{24}$ & 
$1/\np{21361414}$ & $\ff{122970779}{122970786} \sim \np{3.16846}\cdot10^{-8}$ & $\np{0.269781}$ \\
\igual & $1/\np{21361415}$ & $\ff{120445556}{120445564} \sim \np{-5.08256}\cdot10^{-9}$ & $\np{0.257003}$ \\
\igual & $7/\np{149529904}$ & $\ff{137149169}{137149179} \sim \np{1.69887}\cdot10^{-10}$ & $\np{0.238788}$ 
\\
\vv $2^{25}$ & 
$1/\np{42722829}$ & $\ff{250992010}{250992018} \sim \np{1.3301}\cdot10^{-8}$ & $\np{0.258016}$ \\
\igual & 
$1/\np{42722830}$ & $\ff{256042455}{256042463} \sim \np{-5.08256}\cdot10^{-9}$ & $\np{0.251621}$ \\
\igual & 
$3/\np{128168489}$ & $\ff{267001542}{267001550} \sim \np{1.0453}\cdot10^{-9}$ & $\np{0.242399}$ 
\\
\vv $2^{26}$ & 
$1/\np{85445659}$ & $\ff{522185807}{522185816} \sim \np{4.10922}\cdot10^{-9}$ & $\np{0.245319}$ \\
\igual & 
$2/\np{170891319}$ & $\ff{552488478}{552488488} \sim \np{-4.86669}\cdot10^{-10}$ & $\np{0.233456}$ \\
\igual & 
$9/\np{769010935}$ & $\ff{586223936}{586223947} \sim \np{2.3986}\cdot1010^{-11}$ & $\np{0.220238}$ 
\\
\vv $2^{29}$ & 
$1/\np{683565275}$ & $\ff{4662329259}{4662329268} \sim \np{6.62304}\cdot10^{-10}$ & $\np{0.222134}$ \\
\igual &
$1/\np{683565276}$ & $\ff{4743136384}{4743136393} \sim \np{-4.86669}\cdot10^{-10}$ & $\np{0.220568}$ \\
\igual &
$2/\np{1367130551}$ & $\ff{4904750631}{4904750641} \sim \np{8.78178}\cdot10^{-11}$ & $\np{0.212628}$ 
\\
\vv $2^{30}$ & 
$1/\np{1367130551}$ & $\ff{9647887023}{9647887033} \sim \np{8.78178}\cdot10^{-11}$ & $\np{0.208898}$ \\
\igual & 
$6/\np{8202783307}$ & $\ff{10645034813}{10645034824} \sim \np{-7.92992}\cdot10^{-12}$ & $\np{0.199544}$ \\
\igual & 
$7/\np{9569913858}$ & $\ff{10716918381}{10716918392} \sim \np{5.74833}\cdot10^{-12}$ & $\np{0.198424}$ 
\\
\hline
\end{tabular}
\caption{Machin-like formulas 
$2^m\arctan(a_1/b_1) - \arctan(a_2/b_2) = \frac{\pi}{4}$,
corresponding to 
$2^m\arctan(x) - \arctan(R_3(2^m,x)) = \frac{\pi}{4}$,
with $x$ one of the first convergents of the continued fraction of $\pi/2^{m+2}$. 
The notation $f^{r}_{s}$ is used to indicate an irreducible fraction with $r$ digits in the numerator and $s$ digits in the denominator.}
\label{tab:Machin2m}
\end{table}

Then, for fixed $n$ big enough, we can take $x \in \Q$ near $\frac{1}{n} \frac{\pi}{4}$ by using a convergent of the continued fraction of $\frac{1}{n} \frac{\pi}{4}$, and thus we have $a_1/b_1 = x$ and $a_2/b_2 = R_3(n,x)$. In this way, and because $x$ and $R_3(n,x)$ are small numbers, the Lehmer measure of the corresponding Machin-like formula will be small (but the integers $a_2$ and $b_2$ have a lot of digits). 

We can do this with $n=2^m$ and then use \eqref{eq:Rjnm} with $j=3$ to compute $a_2/b_2 = R_3(2^m,x)$. This method is very fast. For $m \le 30$ and using three convergents for every $\frac{1}{2^m} \frac{\pi}{4}$, we have found the corresponding Machin-line formulas, and computed their Lehmer measures. We summarize a collection of these formulas in Table~\ref{tab:Machin2m}.

The Machin-like formulas corresponding to the first convergents of $2^5$ (i.e., $a_1/b_1 = 1/40$) and $2^{26}$ (i.e., $a_1/b_1 = 1/\np{85445659}$) have been previously found in \cite{AbQ, AbSJQ} by a different method.

\section{Machin's formulas with powers of the golden section}
\label{sec:phi}

Recall that $\phi=(1+{\sqrt{5}})/2$ denotes the golden section. There are some linear combinations of arctangents of powers of the golden section which evaluate to a rational multiple of $\pi$ such as
\begin{align*}
\frac{\pi}{4} & = \frac{1}{3}\arctan(\phi^3) + \frac{1}{3}\arctan(\phi) 
= \frac{1}{5}\arctan(\phi^6)+\frac{2}{5}\arctan(\phi^2), \\
\frac{\pi}{4} & = \frac{1}{7}\arctan(\phi^5)+\frac{3}{7}\arctan(\phi^3)
= -\frac{1}{2}\arctan(\phi^5)+\frac{3}{2}\arctan(\phi).
\end{align*}
The first three of them appear for instance in \cite{Chan, LuSt}. The last one, although can be easily obtained from the first three, does not appear in the above papers.
They are all of the form 
\begin{equation}
\label{eq:phikl}
\frac{\pi}{4} = a \arctan(\phi^\kappa)+b\arctan(\phi^{\ell}),
\end{equation}
for positive integers $\kappa>\ell$ with some rational numbers $a,b$. 
Via the formula $\arctan(x)+\arctan(1/x)=\pi/2$ valid for all positive real numbers $x$, each one of the above formulas gives rise to three additional formulas of the same kind with different $(a,b)$, replacing $(\kappa,\ell)$ by $(\pm \kappa,\pm \ell)$. Via the above identity, we see that formula \eqref{eq:phikl} holds as well with $a=b=1/2$, whenever $\kappa+\ell=0$. So, 
eliminating such trivial solutions, we see that equation~\eqref{eq:phikl} holds in $\kappa,\ell \in \Z$, $|\kappa| \ge |\ell|$, $\kappa+\ell \ne 0$ and $a,b \in \Q$ for the following quadruples:
\[
(a, b, \kappa, \ell) \in
\left\{ 
\begin{matrix}
\left(\frac{1}{3}, \frac{1}{3}, 3,1\right), & \left(1,1,-3,-1\right), 
& \left(-1,1,-3,1\right), & \left(1,-1,3,-1\right), \\[4pt]
\left(\frac{1}{5}, \frac{2}{5}, 6,2\right), & \left(1,2,-6,-2\right),
& \left(\frac{-1}{3}, \frac{2}{3},-6,2\right), & \left(1,-2,6,-2\right), \\[4pt]
\left(\frac{1}{7}, \frac{3}{7}, 5,3\right), & \left(1,3,-5,-3\right), 
& \left(\frac{-1}{5}, \frac{3}{5},-5,3\right), & \left(1,-3,5,-3\right), \\[4pt]
\left(\frac{-1}{2}, \frac{3}{2}, 5,1\right), & \left(\frac{-1}{2}, \frac{3}{2},-5,-1\right), 
& \left(\frac{1}{4}, \frac{3}{4},-5,1\right), & \left(\frac{1}{4}, \frac{3}{4},5,-1\right) 
\end{matrix} 
\right\}.
\]
The equation \eqref{eq:phikl} in positive integers $\kappa,\ell$ was treated in \cite{LuSt}. The main result in \cite{LuSt} claims to have found all solutions of equation~\eqref{eq:phikl} in integers $\kappa,\ell$ with $\kappa+\ell\ne 0$. However,~\cite{LuSt} missed the last row of solutions indicated above corresponding to $(\kappa,\ell)=(5,1)$ and its variants with $(\pm 5,\pm 1)$. In this section, we fill in the oversight from \cite{LuSt} and show that there are no other solutions up to signs except for the above four.

Writing as in \cite{LuSt}, $a=u/w$, $b=v/w$ with coprime integers $u,v,w$ and $w\ge 1$, equation \eqref{eq:phikl} leads to 
\begin{equation}
\label{eq:phiuv}
(1+i\phi^\kappa)^{4u} (1+i\phi^{\ell})^{4v} = (1-i\phi^{\kappa})^{4u}(1-i\phi^{\ell})^{4v}
\end{equation}
(formula (4) in~\cite{LuSt}). The norm of the element $1+i\phi^\kappa$ in the biquadratic field $\K = \Q(i,{\sqrt{5}})$ is $5F_\kappa^2$ or $L_\kappa^2$ according to whether $\kappa$ is odd or even, where $F_\kappa$, $L_\kappa$ are the $\kappa$th Fibonacci and Lucas companion of the Fibonacci numbers, respectively. Since the above number is never a power of $2$ for any positive integer $\kappa$, it follows that 
for every odd prime factor $p$ of the above number, there is a prime ideal $\pi$ in $\mathcal{O}_{\K}$ dividing $p$ such that $\pi$ divides $1+i\phi^\kappa$. Note that $\pi$ does not divide 
$1-i\phi^\kappa$, since otherwise $\pi$ divides $(1+i\phi^\kappa)+(1-i\phi^\kappa)=2$, which is false since $\pi$ divides the odd prime~$p$. The same argument applies to $1+i\phi^{\ell}$. 
Using the Primitive Divisor Theorem for Fibonacci and Lucas numbers, it is argued in \cite{LuSt} that $\kappa\le 12$, so one is left with finding all pairs of positive integers $(\kappa,\ell)$ in the range $1<\ell<\kappa\le 12$. Then in \cite{LuSt} (see formula~(5)) it is said that $\pi$ divides $1+i\phi^\kappa$ and $1-i\phi^{\ell}$ and it is shown that, under this assumption, $(\kappa,\ell) = (6,2)$, $(5,3)$, $(5,1)$. Looking at formula (4) in \cite{LuSt} (or formula \eqref{eq:phiuv} above) however, the assumption that $\pi$ divides $1+i\phi^\kappa$ and $1-i\phi^{\ell}$ implies that $u$ and $v$ \textbf{have the same sign}. In fact the solutions from \cite{LuSt} have $a$ and $b$ with the same sign. Thus, the oversight comes from not having treated the case when $u$ and $v$ \textbf{have opposite signs} in~\cite{LuSt}. In this case, $\pi$ 
divides $1+i\phi^\kappa$ and $1+i\phi^{\ell}$. This is the case missed in~\cite{LuSt}. At any rate, all examples must satisfy that the set of odd prime factors of the two numbers 
\[
N_{\K/\Q}(1+i\phi^\kappa)
\quad \text{and}\quad 
N_{\K/\Q}(1+i\phi^{\ell})
\]
must be the same. One calculates all such numbers for $1\le \ell<\kappa\le 12$ and 
gets the four solutions $(\kappa,\ell) = (3,1)$, $(5,1)$, $(5,3)$, $(6,2)$ and no others.

\section*{Acknowledgements}
The first author is partially supported by
the Minis\-te\-rio de Cien\-cia e Inno\-va\-ci\'{o}n (PID2019-104658GB-I00 grant), 
by the grant Severo Ochoa
and Mar\'ia de Maeztu Program for Centers and Units of Excellence in
R\&D (CEX2020-001084-M)
and also by the Ag\`{e}ncia de Gesti\'{o} d'Ajuts Universitaris i de Recerca (2021 SGR 113 grant). 
The second author worked on this paper during a visit at the Max Planck Institute for Software Science in Saarbr\"ucken, Germany, in Spring of 2022. He thanks the people of this Institute for hospitality and support.
This author was also partially supported by Project 2022-064-NUM-GANDA from the CoEMaSS at Wits.
The third author is partially supported by
the Minis\-te\-rio de Cien\-cia e Inno\-va\-ci\'{o}n (PID2021-124332NB-C22 grant).

The authors want to acknowledge the helpful, constructive and detailed revision of the manuscript by the referees, that has allowed us to correct some errors and to improve the final version of the paper.



\end{document}